\documentclass[times,sort&compress,3p]{elsarticle}
\journal{Journal of Multivariate Analysis}
\usepackage[labelfont=bf]{caption}

\newcommand{\bT}{\mathbb{T}}
\newcommand{\bA}{\mathbb{A}}
\newcommand{\bP}{\mathbb{P}}
\newcommand{\bE}{\mathbb{E}}
\newcommand{\bG}{\mathbb{G}}
\newcommand{\bF}{{\mathbb{F}}}
\newcommand{\bH}{{\mathbb{H}}}
\newcommand{\bI}{{\mathbb{I}}}

\newcommand{\Sg}{{\rm{Sign}}}
\newcommand{\tr}{{\rm{Trace}}}

\newcommand{\bsX}{{\boldsymbol{X}}}

\usepackage{amsmath,amsfonts,amssymb,booktabs,color,epsfig,graphicx,hyperref,url}

\usepackage{amsthm,graphicx,natbib,ulem,hyperref}
\usepackage{enumitem}
\usepackage{multirow}

\theoremstyle{plain}
\newtheorem{theorem}{Theorem}

\newtheorem{lemma}{Lemma}

\theoremstyle{definition}

\newtheorem{remark}{Remark}
\newtheorem{example}{Example}

\def\red{\textcolor{red}}

\def\ban#1\ean{\begin{align}#1\end{align}}
\def\bans#1\eans{\begin{align*}#1\end{align*}}

\newcommand{\mbf}[1]{\mbox{\boldmath $#1$}}
\newcommand{\ba}{{\mbf \beta}}

{\catcode `\@=11 \global\let\AddToReset=\@addtoreset}
\AddToReset{equation}{section}

\AddToReset{Theorem}{section}

\newtheorem{cor}{Corollary}[section]

\def\ba{\begin{array}}
\def\bc{\begin{center}}
\def\bd{\begin{description}}
\def\be{\begin{enumerate}}
\def\ea{\end{array}}
\def\ec{\end{center}}
\def\ed{\end{description}}
\def\edt{\end{document}}
\def\ee{\end{enumerate}}
\def\ben{\begin{equation}}
\def\benn{\begin{equation*}}
\def\een{\end{equation}}
\def\eenn{\end{equation*}}
\def\benr{\begin{eqnarray}}
\def\eenr{\end{eqnarray}}
\def\benrr{\begin{eqnarray*}}
\def\eenrr{\end{eqnarray*}}

\def\r{\ref}

\def\vep{\varepsilon}

\def\bI{{\mbf I}}

\def\bX{{\mbf X}}

\def\bZ{{\mbf Z}}

\def\bal{\begin{align*} }
\def\eal{\end{align*} }
\def\baln{\begin{align} }
\def\ealn{\end{align}}

\begin{document}

\def\bvep{{\mbf \vep}}
\def\red{\textcolor{red}}
\def\blu{\textcolor{blue}}
\def\grn{\textcolor{green}}
\def\ve{{\rm vec}}
\def\1{{\mbf 1}}
\def\bfe{{\mbf e}}

\begin{frontmatter}

\title{Limiting Spectral Distribution of moderately large Kendall's correlation matrix and its application\\}

\author[1] {Raunak Shevade}\corref{mycorrespondingauthor}
\author[1]{Monika Bhattacharjee}

\address[1]{Department of Mathematics, Indian Institute of Technology Bombay}

\cortext[mycorrespondingauthor]{Corresponding author: Raunak Shevade, Email address: \url{shevade.raunak@gmail.com}}

\begin{abstract}
We establish the limiting spectral distribution of Kendall’s correlation matrices in the moderate high-dimensional regime where the dimension grows slower than the sample size. Our framework allows observations to be independent but not necessarily identically distributed, and accommodates both discrete and continuous data. Unlike existing results developed under i.i.d.\ observations, our approach remains valid under substantial distributional heterogeneity and also covers certain i.i.d.\ models beyond previously studied settings.

Under mild symmetry and convergence conditions on some traces, we prove that the empirical spectral distribution of a properly centered and scaled Kendall’s correlation matrix converges weakly almost surely to a deterministic, generally model-dependent limit. The analysis clarifies how distributional heterogeneity influences the limiting spectrum.

As an application, we propose a graphical tool for detecting dependence among components in high-dimensional data and show that ignoring heterogeneity may lead to spurious detection of dependence.
\end{abstract}


\end{frontmatter}

\section{Introduction} \label{sec: sintro}

Sample covariance and correlation matrices play a central role in multivariate statistical analysis. In high-dimensional settings, a common strategy for understanding the interdependence among variables involves studying the asymptotic behavior of the eigenvalues and eigenvectors of these matrices. This area has been extensively explored within the framework of random matrix theory.

Consider a random matrix $\mathbb{R}_p$ of dimension $p \times p$. The empirical spectral distribution (ESD) of $\mathbb{R}_p$ assigns a mass of $1/p$ to each of its eigenvalues and is denoted by $\nu^{\mathbb{R}_p}$. Its corresponding cumulative distribution function is $F^{\mathbb{R}_p}$, also referred to as the empirical cumulative distribution function (ECDF). If $F^{\mathbb{R}_p}$ converges weakly (in probability or almost surely) to a limit as $p \to \infty$, the limiting distribution is called the \textit{limiting spectral distribution} (LSD). For a function \( f \), the associated \textit{linear spectral statistic} (LSS) is defined as $\int f(x)\, dF^{\mathbb{R}_p}(x)$. See \citet{bai2010spectral} for a comprehensive account of LSDs and LSSs of sample covariance matrices, and \citet{jiang2004limiting, gao2017high} for similar studies involving sample correlation matrices.

To ensure almost sure convergence to an LSD, it is common to assume that the second moments of the entries are finite. However, when the data have heavy tails, many of the established results for light-tailed distributions break down. For instance, \citet{heiny2022limiting} demonstrated that the spectral behavior of sample correlation matrices in heavy-tailed scenarios deviates significantly from that of the light-tailed case.

In such heavy-tailed contexts, robust, non-parametric alternatives to sample covariance matrices, such as those based on ranks, are often preferred. In particular, Kendall’s correlation and Spearman’s $\rho$ matrices have garnered significant attention due to their favorable asymptotic properties even without strong moment assumptions. For example, \citet{ECP} and \citet{bai2008large} established that, under suitable conditions (e.g., $p, n \to \infty$ with $p/n \to \theta \in (0, \infty)$), the LSDs of Kendall’s correlation and Spearman’s $\rho$ matrices follow the Mar\v{c}enko–Pastur law or its affine transformation. Asymptotic results for LSSs and extremal eigenvalues of these matrices have also been studied in works such as \citet{bao2015spectral, bao2019tracya, bao2019tracyb, li2021central}. However, these results typically rely on the assumption that the data are identically distributed, absolutely continuous and $p$ grows at the same rate with sample size.
Most recently, \citet{dornemann2025ties} investigated the LSD of properly centered and normalized Kendall’s correlation and Spearman’s $\rho$ matrices under i.i.d.\ observations that may be discrete, allowing for the general regime $p/n \to \theta \in [0,\infty)$. In a different asymptotic regime, \citet{bousseyroux2026another} studied the LSD of Kendall’s correlation matrix when $p/n^2$ converges to a positive constant. \citet{kendep2023} studies the LSD of Kendall’s correlation matrix under dependence among components within i.i.d.\ continuous observations in the proportional high-dimensional regime. 


This work investigates the relatively less-explored setting where the observations may be non-identically distributed continuous and/or discrete. Our focus is on the limiting spectral behavior of Kendall’s correlation matrix in a \textit{moderate high-dimensional} regime, where the dimension $p$ grows slower than the sample size $n$, that is, $p/n \to 0$.

The moderate high-dimensional regime is well-established and frequently encountered in the literature on high-dimensional statistics (see, e.g., \citet{WP2014}, \citet{WAP2015}, \citet{bhattacharjee2019joint}). Results derived for the proportional growth regime $p/n \to \theta \in (0,\infty)$ generally degenerate to non-informative limits when $\theta$ is naively set to zero. Consequently, such results cannot be directly applied for statistical inference in the $\theta = 0$ case. Obtaining non-degenerate, informative limits requires different centering and scaling, and the analysis becomes fundamentally different from the proportional growth setting. Indeed, $\theta = 0$ and $\theta \in (0,\infty)$ typically yield qualitatively distinct results; see, for example, \citet{bhattacharjee2016large} and \citet{bhattacharjee2019joint}, or \citet{LAP2013} and \citet{WAP2015}, where the LSDs are derived separately for each regime. In this paper, we establish the LSD of Kendall’s correlation matrix in the case $p/n \to 0$.

Kendall’s correlation matrix is constructed from pairwise comparisons using a sign-based kernel. The sign function is defined as \(
\mathrm{Sign}(x)=\frac{x}{|x|}\bI(x \neq 0)
\),  where $\bI(\cdot)$ is the indicator function. 
For vectors in $\mathbb{R}^2$, define the kernel
\(
h((x_1, y_1), (x_2, y_2)) = \text{Sign}(x_1 - x_2) \, \text{Sign}(y_1 - y_2).
\)
Given $n$ paired observations $\{(x_i, y_i)\}_{i=1}^n$, Kendall’s $\tau$ is the U-statistic with kernel $h$:
\[
\binom{n}{2}^{-1} \sum_{1 \leq i < j \leq n} h((x_i, y_i), (x_j, y_j)).
\]
Let $X = (X_{ki})$ be a $p \times n$ data matrix. For each pair $(k, l)$, define $T_{kl}$ as the Kendall’s $\tau$ between the sequences $\{X_{ki}\}_{i \in [n]}$ and $\{X_{li}\}_{i \in [n]}$ where $[n] = \{1,2,\ldots, n\}$. The Kendall’s correlation matrix is then $\bT = (T_{kl})_{1 \leq k,l \leq p}$. 

The first result on the LSD of Kendall’s correlation matrix for independent observations appears in \citet{ECP}, where the columns of $X$ are assumed to be identically distributed with a density. They show that the LSD of $\bT$ converges in probability to the distribution of $\frac{2}{3} M_\theta + \frac{1}{3}$, where $M_\theta$ follows the Mar\v{c}enko–Pastur law with parameter $\theta$. 
More recently, \citet{dornemann2025ties} established that, under i.i.d.\ observations that may be discrete, the LSD of a properly centered and normalized Kendall’s correlation matrix is an affine transformation of the Mar\v{c}enko–Pastur law in the proportional high-dimensional regime, and converges to the semi-circle law in the moderate high-dimensional regime. 
In a different asymptotic regime, where the dimension is proportional to the square of the sample size, \citet{bousseyroux2026another} showed that, for i.i.d.\ continuous observations, the Hoeffding decomposition becomes negligible, and instead a component of the remainder term governs the limiting behavior.

Let ${\rm D}(\bA)$ denote the diagonal matrix formed from the diagonal entries of $\bA$. For $\bT$, the $k$-th diagonal entry is
\(
(n(n-1))^{-1} \sum_{i \neq j} \bI\{X_{ki} \neq X_{kj}\}.
\)
When the data is absolutely continuous, this value is almost surely 1, so ${\rm D}(\bT) = \bI_p$, the identity matrix of order $p$ and $\bT - {\rm D}(\bT)$ differs from $\bT$ only by a constant shift. In contrast, when absolute continuity does not hold, the diagonal entries may vary substantially and are no longer uniformly 1. This variability complicates direct analysis of $\bT$. Since the diagonals correspond to self-association—which is often less informative—we instead focus on the centered matrix $\bT - {\rm D}(\bT)$. This approach accommodates both discrete and heterogeneous data.

Our main result establishes the almost sure weak convergence of the ESD of $\bT - {\rm D}(\bT)$ under three assumptions. The first, as in \citet{ECP}, is the independence of the entries of $X$. The second requires a symmetry condition on the distribution of $\mathrm{Sign}(X_{ki}-X_{kj})$ and its reverse for all $i\neq j$; while this holds automatically in the i.i.d.\ setting, it also accommodates a much broader class of non-identically distributed observations. The third assumption imposes conditions on the traces of powers of certain variance--covariance matrices $\mathbb{G}_{k,i}$, defined via conditional expectations of sign functions. Specifically, we require that the averaged quantities $n^{-1}\mathrm{Tr}(\mathbb{G}_{k,i})$ and $n^{-2}\mathrm{Tr}(\mathbb{G}_{k_1,i}\mathbb{G}_{k_2,i})$ over $i\in[n]$ converge to constants at an appropriate rate. Notably, these conditions are strictly weaker than those in \citet{ECP} and, crucially, allow for heterogeneous (non-identically distributed) observations.

Under these assumptions, Theorem~\ref{thm: A} shows that the LSD of a suitably centered and scaled version of $\bT - {\rm D}(\bT)$ exists almost surely and is characterized by a symmetric probability distribution. In general, this limit is not the semi-circle law (see Examples~\ref{Example: A} and \ref{Example: B}). In the special case of independent and identically distributed observations, the LSD admits a characterization via the free multiplicative convolution of a semi-circle law and the LSD of an associated variance--covariance matrix of bounded transformations of the data (whenever the latter exists). Furthermore, Theorem~\ref{thm: 1} identifies a class of non-identically distributed data matrices with controlled component-wise heterogeneity for which the LSD reduces to the semi-circle law. Section~\ref{sec: example} presents three illustrative examples.  In Section \ref{sec: adhoc}, we briefly discuss the idea of a data-driven ad-hoc method for verifying our crucial assumptions. A rigorous mathematical justification of this approach requires more advanced tools and techniques and is left for future work.

It is important to emphasize that \citet{dornemann2025ties}, the only existing work in a comparable framework, focuses on a normalized Kendall’s correlation matrix. In contrast, we analyze a centered and scaled version, and this difference is not merely technical: it leads to substantially different asymptotic behavior  and allows us to accommodate some models even in i.i.d. framework outside the scope of \citet{dornemann2025ties}. 
The two approaches coincide only in the classical setting of continuous i.i.d.\ observations, where both sets of assumptions are automatically satisfied. In more general settings, particularly in the presence of discrete or zero-inflated data, the asymptotic behaviors may diverge even under i.i.d.\ sampling. The normalization in \citet{dornemann2025ties} excludes the presence of asymptotically degenerate components, an assumption that may be unrealistic in applications involving sparse or zero-inflated data.  In contrast, our framework accommodates such degeneracies (see Example~\ref{Example: B}), where the normalization-based approach fails but our assumptions remain valid and the LSD exists.

On the other hand, the absence of normalization in our setting introduces heterogeneity in the scale of different components, necessitating explicit control through our trace conditions. This stands in contrast to \citet{dornemann2025ties}, where normalization effectively homogenizes the variability across components. Consequently, their results (e.g., Theorem~2.5(1)) and our Theorems~\ref{thm: A} and \ref{thm: 1} are theoretically distinct and not directly comparable. Nevertheless, we note that, with suitable modifications, our proof techniques can be adapted to recover their main result (Remark \ref{rem: heinypf}). 

Most importantly, we emphasize that our work provides a first systematic step toward developing a theory for Kendall’s correlation matrices under non-identically distributed observations. Notably, all examples presented in Section~\ref{sec: 2}, involving non-identically distributed data, satisfy our assumptions while remaining outside the scope of the framework developed in \citet{dornemann2025ties}, thereby illustrating the broader applicability of our results.

Handling non-identically distributed observations is therefore a central contribution, but not the sole source of novelty. Even within certain i.i.d.\ models, there exist data-generating mechanisms that fall outside the assumptions of \citet{dornemann2025ties} yet are covered by our framework (see Example~2). The present manuscript thus offers not only a heterogeneous extension, but also a conceptually distinct analytical approach that captures regimes not addressed by existing results.


We further investigate the implications of our findings for testing independence in high-dimensional settings.  
Classical approaches—such as likelihood ratio tests (\citet{jiang2013testing}), $L^2$-norm based methods (\citet{gao2017high, yang2015independence}), and $L^\infty$-norm based methods (\citet{jiang2004asymptotic, zhou2007asymptotic, cai2011limiting})—often require finite high-order moments and are therefore unsuitable for heavy-tailed data.  

In contrast, rank-based methods, including those based on Kendall’s correlation and Spearman’s $\rho$, are more robust under heavy-tailed distributions. Significant contributions in this direction include $L^2$-norm tests using the Spearman’s matrix (\citet{bao2015spectral}), test based on LSS of  Kendall's correlation matrix (\citet{li2021central}), $L^\infty$-norm tests for both Kendall’s and Spearman’s matrices (\citet{han2017distribution}), and general $U$-statistic based tests (\citet{leung2018testing}) under the assumption of independent, absolutely continuous and sometimes symmetric data with $p/n \to \theta \in (0,\infty)$.  

Theorem~\ref{thm: 1} provides a foundation for graphical diagnostics to detect row-wise dependence in data matrices, even when the underlying distributions are discrete or heterogeneous across columns and $p/n \to 0$. We illustrate, through a simple empirical study, that ignoring heterogeneity in the distributions of the data matrix can lead to misleading conclusions, in particular, spurious detection of dependence. This highlights the importance of explicitly accounting for non-identically distributed observations in high-dimensional independence testing. A comparison with an alternative approach based on \citet{dornemann2025ties} further indicates that methods designed under homogeneous assumptions may suffer from size distortions and reduced power in heterogeneous settings.  Details of this diagnostic are given in Section~\ref{sec: application}.  
While a formal significance test could be constructed by applying an appropriately chosen LSS of $\bT - {\rm D}(\bT)$, deriving the corresponding asymptotic distribution lies beyond the scope of the present work.

\section{Convergence of empirical spectral distributions} \label{sec: 2}  
\noindent We make the following assumptions on the data $\{X_{ki}: 1\leq k \leq p, 1\leq i\leq n\}$. 
\vskip 5pt
\noindent \textbf{Assumption 1}: For all $p,n \geq 1$, the collection $\{X_{ki}: 1\leq k \leq p, 1\leq i\leq n\}$ consists of independent random variables.
\vskip 3pt
\noindent Assumption 1 is crucial for our proof techniques. The dependent case requires different methods and will be investigated in future work. Some partial results for dependent settings are available in \citet{kendep2023}, though further generalization is needed.
\vskip 3pt
\noindent \textbf{Assumption 2}:   We have $\bP(X_{ki}>X_{kj}) = \bP(X_{ki}<X_{kj})$ for all  $1\leq k \leq p, 1\leq i,j\leq n$ and $p,n \geq 1$.
\begin{remark} \label{rem: A2} Assumption 2 is required to ensure that $\mathbb{E}(\Sg(X_{ki}-X_{kj})) = 0$ for all $k,i,j$. 
Note that for each fixed $k$, if the collection $\{X_{ki}\}_{i=1}^n$ is identically distributed across $i$, then Assumption 2 is automatically satisfied. Indeed, $X_{ki}-X_{kj}$ and $X_{kj}-X_{ki}$ are identically distributed, which implies that the distribution of $X_{ki}-X_{kj}$ is symmetric about zero. Consequently,
\(
\mathbb{E}\big(\Sg(X_{ki}-X_{kj})\big)=0 .
\)

Note that Assumption 2 does not require each $X_{ki}$ to have a symmetric distribution. For example, let $X$ and $Y$ be independent with 
\(
\mathbb{P}(X=-1)=\mathbb{P}(X=0.5)=\mathbb{P}(Y=-2)=\mathbb{P}(Y=3)=0.5.
\)
Then $\mathbb{P}(X>Y)=\mathbb{P}(X<Y)=0.5$, although neither $X$ nor $Y$ is symmetrically distributed.

On the other hand, even if each $X_{ki}$ is symmetrically distributed about its median, Assumption 2 need not hold. For instance, let $X$ and $Y$ be independent uniform random variables on $(-1,1)$ and $(2,3)$, respectively. Then $\mathbb{P}(X<Y)=1$ and $\mathbb{P}(X>Y)=0$.

However, if for each $k$ the variables $\{X_{ki}\}_{i=1}^n$ are symmetric about a common median (say $0$, without loss of generality) across all $i$, then Assumption 2 holds. In this case $X_{ki}$ and $-X_{ki}$ are identically distributed, and therefore
\(
\mathbb{P}(X_{ki}>X_{kj})=\mathbb{P}(-X_{ki}>-X_{kj})=\mathbb{P}(X_{ki}<X_{kj})
\)
for all $i\neq j$.

\end{remark}
Since each $T_{kl}$ is a U-statistic, its asymptotic behaviour is governed by the Hoeffding decomposition. In the high-dimensional regime considered here, the first-order (linear) projection dominates, and therefore determines the limiting spectral distribution. 
To define the Hoeffding projection, let
 $F_{kj}(t) = \bP(X_{kj}\leq t) \ \text{and}\ \bar{F}_{kj}(t) = \bP(X_{kj} < t), \ \text{for all}\ \ t \in \mathbb{R}.$  Let $Y_{k,i,j}$ be the conditional expectation of $\Sg (X_{ki}-X_{kj})$ given $X_{ki}$, i.e., for all $1 \leq i \neq j \leq n$ and $1 \leq k \leq p$, we have 
\begin{eqnarray}
Y_{k,i,j} = \bE(\Sg(X_{ki}-X_{kj})|X_{ki})
&=& \bP(X_{ki}>X_{kj}| X_{ki}) - \bP(X_{ki} < X_{kj}|X_{ki})
= \bar{F}_{kj}(X_{ki})  + F_{kj}(X_{ki})-1. \nonumber 
\end{eqnarray}
\noindent The following properties  of $Y_{k,i,j}$, are immediate from Assumption 1.
\vskip 3pt

\noindent (1)  \ $Y_{k,i,j}$ is a measurable function of $X_{ki}$.
\quad  (2) \ $Y_{k,i,j}$ is independent of $\{X_{lj}: \ 1 \leq l \leq  p, \ 1 \leq j \leq n, 
\ l \neq k,\ j \neq i\}$. 
\quad 
 (3) \ $\{Y_{k,i,j}\}$ are independent across $1 \leq k \leq p$,   $1 \leq i \leq  n$. 
\vskip 3pt
\noindent For each $1 \leq k \leq p$ and $1 \leq i \leq n$, define the variance--covariance matrix of  $(Y_{k,i,j}:\ 1 \leq j \leq n)$ by the $n \times n$ matrix
\begin{eqnarray}
\mathbb{G}_{k,i} =
\big(\mathrm{Cov}(Y_{k,i,j_1}, Y_{k,i,j_2})\big)_{1 \leq j_1, j_2 \leq n}.
\label{eqn: defGki} \nonumber 
\end{eqnarray}

\noindent The first-order Hoeffding projection of $\mathbb{T}$ is given by the $p \times p$ matrix $G$ whose $(k,l)$-th entry is
\begin{eqnarray}
G_{kl} = \frac{1}{n(n-1)} \sum_{i \neq j} Y_{k,i,j}\, Y_{l,i,j}. \nonumber
\end{eqnarray}
\begin{remark} \label{rem: GG}
Suppose $\{X_{ki}\}$ satisfy Assumption 1 and are identically distributed across $i$. Then  $Y_{k,i,j}$ does not depend on $j$, so we may write $Y_{k,i,j} = Y_{ki}$ (say) for all $j$. Further, for each fixed $k$, the collection $\{Y_{ki}\}_{i=1}^n$ is identically distributed. Consequently, for each $k$, the matrix $\mathbb{G}_{k,i} $ reduces to
\(
\mathbb{G}_{k,i}^{\rm{IID}} = \mathrm{Var}(Y_{k1}) \, \mathbb{J}_n,
\)
where $\mathbb{J}_n$ denotes the $n \times n$ matrix with all entries equal to $1$.  Moreover, $G$ reduces to $G^{\rm{IID}}$ whose $(k,l)$-th entry is $G_{kl}^{\rm{IID}} = \frac{1}{n} \sum_{i =1}^n Y_{ki}\, Y_{li}$. Further, $\{Y_{ki}\}$'s are independent across both $k$ and $i$. 
\end{remark}

\noindent In Section~\ref{sec: hoeffneg}, we show that, after appropriate centering and scaling, the ESD of $\mathbb{T}-D(\mathbb{T})$ and its first-order projection $G$ coincide in the limit as $p,n \to \infty$ with $p/n \to 0$. Consequently, the limiting spectral distribution is determined by the asymptotic behaviour of the collection of covariance matrices $\{\mathbb{G}_{k,i}\}$. The following Assumptions~G1 and G2 provide the necessary control on these matrices.
Let $\mathrm{NC}_2(2R)$ denote the set of all non-crossing pair partitions of $[2R]$. Define $\sigma_R(i)=i+1$ for $1\le i\le R-1$ and $\sigma_R(R)=1$, let $\delta_{ab}$ be the Kronecker delta, that is, $\delta_{ab}=1$ if $a=b$ and $0$ otherwise.


\vskip 5pt
\noindent \textbf{Assumption G1.} There exists $g_1 \in \mathbb{R}$ such that
\(\max(\sqrt{np},(n/p))\left( n^{-2}p^{-1}\sum_{k,i} \tr(\mathbb{G}_{k,i}) - g_1\right) \to 0.\) 

\vskip 5pt
\noindent \textbf{Assumption G2.} There exist non-negative constants $g_{2\pi}$ such that, for all $R \ge 1$ and $\pi \in \mathrm{NC}_2(2R)$,
\begin{equation}
(n^{-R}p^{-1})\sum_{k_1,\ldots,k_{2R}}
\left(
\prod_{(b,s)\in \pi}
\Big(
p^{-1}n^{-2}\sum_{i_b}\tr(\mathbb{G}_{k_b,i_b}\mathbb{G}_{k_{\sigma_{2R}(b)},i_b})
\,
\delta_{k_b,k_{\sigma_{2R}(s)}}
\delta_{k_s,k_{\sigma_{2R}(b)}}
\delta_{i_b,i_s}
\Big)
\right)
\to g_{2\pi}. \label{eqn: G2g}
\end{equation}
\noindent Clearly, $g_{2\pi} \le C^R$ for some constant $C>0$ and for all $\pi \in \mathrm{NC}_2(2R)$. Hence, by Carleman's condition, the sequence of moments $\sum_{\pi \in \mathrm{NC}_2(2R)} g_{2\pi}$ uniquely determines a probability distribution.


\begin{remark} \label{rem: GGG}
Suppose $\{X_{ki}\}$ satisfy Assumption 1 and are identically distributed across $i$. Then
Assumption G1 is satisfied provided there exists a constant $g_1$ such that
\begin{eqnarray}
\max(\sqrt{np},(n/p))\left( \frac{1}{p}\sum_{k=1}^p \mathrm{Var}(Y_{k1}) - g_1 \right) \to 0. \label{eqn: g1A}
\end{eqnarray}
Using $\mathbb{G}_{k,i} = \mathrm{Var}(Y_{k1}) \mathbb{J}_n$ and $\tr(\mathbb{J}_n^2)=n^2$, the left side of (\ref{eqn: G2g}) reduces to
\begin{eqnarray}
&& p^{-(R+1)}\sum_{k_1,\ldots,k_{2R}}
\prod_{(b,s)\in \pi}
\Big(
\mathrm{Var}(Y_{k_b1})\, \mathrm{Var}(Y_{k_{\sigma_{2R}(b)}1})
\,
\delta_{k_b,k_{\sigma_{2R}(s)}}
\delta_{k_s,k_{\sigma_{2R}(b)}}
\Big).
\label{eqn:abG_reduced}
\end{eqnarray}
For any $\pi \in \mathrm{NC}_2(2R)$, the number of distinct indices in
\(
\{(k_1,\ldots,k_{2R}) : \delta_{k_b,k_{\sigma_{2R}(s)}} \delta_{k_s,k_{\sigma_{2R}(b)}}=1, \ (b,s)\in \pi\}
\)
is exactly $R+1$. Hence, the sum in \eqref{eqn:abG_reduced} factorizes into products of terms of the form
\(
\frac{1}{p}\sum_{k=1}^p \big(\mathrm{Var}(Y_{k1})\big)^r,
\)
for various integers $r \geq 1$ determined by the structure of $\pi$.
For example, when $R=4$:
\begin{itemize}
\item The pairing $\{(1,2),(3,4),(5,6),(7,8)\}$ yields
\(
\left(\frac{1}{p}\sum_{k=1}^p (\mathrm{Var}(Y_{k1}))^4\right)
\left(\frac{1}{p}\sum_{k=1}^p \mathrm{Var}(Y_{k1})\right)^4,
\)
\item The pairing $\{(1,4),(2,3),(5,8),(6,7)\}$ yields
\(
\left(\frac{1}{p}\sum_{k=1}^p (\mathrm{Var}(Y_{k1}))^2\right)^3
\left(\frac{1}{p}\sum_{k=1}^p \mathrm{Var}(Y_{k1})\right)^2.
\)
\end{itemize}
Thus, Assumption G2 holds provided the following condition is satisfied.
\vskip 5pt
\noindent \textbf{Assumption G2-a}: The ESD of the diagonal matrix
\(
\Sigma_p = \mathrm{diag}\big(\mathrm{Var}(Y_{11}), \ldots, \mathrm{Var}(Y_{p1})\big)
\)
converges weakly to a limiting probability distribution. 
\end{remark}

\vskip 5pt
\begin{theorem}\label{thm: A}
Suppose Assumptions~1, 2, G1, and G2 hold, and $p := p(n) \to \infty$ with $p/n \to 0$ as $n \to \infty$. Then 
the ESD of 
\(
\sqrt{n/p}\big(\mathbb{T}-D(\mathbb{T})\big)
\) or $2\sqrt{n/p}(G - g_1\mathbf{I}_p)$ 
converges weakly, almost surely, to a probability distribution whose odd-order moments vanish and whose $2R$-th moment is given by
\(
2^{2R}\sum_{\pi \in \mathrm{NC}_2(2R)} g_{2\pi}.
\)
\end{theorem}

\vskip 5pt
\noindent Next, recall the semi-circle law with probability density function
\begin{equation}
f(x) =
\begin{cases}
\dfrac{2}{\pi R^2}\sqrt{R^2 - x^2}, & -R < x < R,\\
0, & \text{otherwise},
\end{cases} \nonumber 
\end{equation}
which we denote by $S_R$.
In general, the LSD obtained in Theorem~\ref{thm: A} does not reduce to the semi-circle law. Such a reduction occurs only in the special case where $g_{2\pi} = C^R$ for some constant $C>0$ and for all $\pi \in \mathrm{NC}_2(2R)$, in which case the limit coincides with the semi-circle law $2S_{2C}$.

As mentioned in Section \ref{sec: sintro}, to the best of our knowledge, \citet{dornemann2025ties} is the only existing work to compare. However, their analysis is restricted to independent and identically distributed observations, allowing for both discrete and continuous distributions, and does not cover the case of non-identically distributed data.

We now present an example involving observations that are non-identically distributed across both rows and columns. This example demonstrates that our results remain valid in this more general setting, where the results of \citet{dornemann2025ties} are not applicable.
\begin{example}\label{Example: A}
Let $X_{ki}\sim \mathrm{Cauchy}(0,\sigma_{ki})$ with $\sigma_{ki}=1$ if $k,i$ are odd, $10$ if both are even, $20$ if $k$ is odd and $i$ is even, and $30$ if $k$ is even and $i$ is odd; thus $\{X_{ki}\}$ are non-identically distributed in both $k$ and $i$. Moreover, $\{X_{ki}\}$ are independent across both indices, and hence Assumption~1 is satisfied. Since $X_{ki}-X_{kj}$ is symmetric about $0$, Assumption~2 is also satisfied. It is easy to see that, since only finitely many distinct probability distributions are involved in the data matrix, Assumptions G1 and G2 are satisfied. The first two rows of Table~\ref{tab:Table1aA} agree, showing that the theoretical moments derived from Theorem~\ref{thm: A} match well with the simulated moments for this example. In contrast, the last two rows of Table~\ref{tab:Table1aA} do not agree, indicating that the results of \citet{dornemann2025ties} are not applicable in this setting.

\begin{table}[h!]
\centering
\caption{Moment comparison for Example~\ref{Example: A} with $n=900$ and $p=30$.}
\label{tab:Table1aA}
\begin{tabular}{|c|c|c|c|c|c|c|}
\hline
Order of moments & $r=2$ & $r=3$ & $r=4$ & $r=5$ & $r=6$\\
\hline
Theoretical moment based on our result &  1.08 & 0  & 8.97 & 0 & 14.58\\
\hline
Empirical moment of ECDF of $(3/2)\sqrt{n/p}\big(\mathbb{T}-D(\mathbb{T})\big)$ & 0.97  & 0  & 9.12 & 0 & 14.23 \\
\hline
Empirical moment of ECDF of $(3/2)\sqrt{n/p}\mathbf{T}$ where & 1.01 & 0 & 9.27 & 0 & 14.28 \\
 $\mathbf{T}$ is defined after the equation (6) in \citet{dornemann2025ties} &&&&& \\
 \hline
 Standard semi-circle moments & 1 & 0 & 2& 0 & 5 \\
 \hline
\end{tabular}
\end{table}
\end{example}

\subsection{IID observations}
In this section, we focus on the special case of independent and identically distributed observations and compare our findings with those of \citet{dornemann2025ties}, which, to the best of our knowledge, contains results closest to our setting.
Following corollary is immediate from Remarks \ref{rem: A2} - \ref{rem: GGG} and (22.10) in \citet{NS2006}. 

\begin{cor} \textbf{Identically distributed observations}. \label{cor: IID}
Suppose $\{X_{ki}\}$ satisfy Assumption 1 and are identically distributed across $i$. Further suppose (\ref{eqn: g1A}) and Assumption G2-a are satisfied, and $p := p(n) \to \infty$ with $p/n \to 0$ as $n \to \infty$. Then 
the ESD of 
\(
\sqrt{n/p}\big(\mathbb{T}-D(\mathbb{T})\big)
\) or $2\sqrt{n/p}(G^{\rm{IID}} - g_1\mathbf{I}_p)$ 
converges weakly, almost surely, to $2S_2 a$ where $S_2$ and $a$ are free.
\end{cor}

\begin{cor}  \textbf{Identically distributed continuous observations}. \label{cor: IID1}
Suppose $\{X_{ki}\}$'s are continuous random variables,  satisfy Assumption 1 and are identically distributed across $i$. Then $\{Y_{ki}\}$'s are all independent uniform random variables on $(-1,1)$. Hence $\mathbb{V}{\rm{ar}}(Y_{ki}) = 1/3$ for all $k,i$. Thus $(\ref{eqn: g1A})$ holds true for $g_1 = 1/3$ and Assumption G2-a are satisfied as the LSD of $\Sigma_p$ exists and is degenerated at $1/3$. Also note that $D(\mathbb{T}) = \mathbf{I}_p$. Therefore, the ESD of  \(\sqrt{n/p}\big(\mathbb{T}-\mathbf{I}_p\big)\) converges weakly, almost surely, to $(2/3)S_2$ which is  Theorem 2.5 (1) of \citet{dornemann2025ties} for continuous observations. 
\end{cor}

\begin{remark} \textbf{Comparison with \citet{dornemann2025ties}}. \label{rem: heiny}
Suppose $\{X_{ki}\}$ satisfy Assumption 1 and are identically distributed across $i$.
\citet{dornemann2025ties} considered the matrix $\mathbf{T}$ defined after equation (6) in their paper. In their Sections 4.4 and 4.5—analogous to our Section 4.1.2, but requiring additional arguments to control the normalization factor via the Dvoretzky--Kiefer--Wolfowitz inequality—they showed that the LSD of $\sqrt{n/p}\mathbf{T}$ coincides with that of $(2/3)\sqrt{n/p}(\tilde{G}^{\mathrm{IID}}-\bI_p)$. The matrix $\tilde{G}^{\mathrm{IID}}$ has zero diagonal entries, and for $k \neq l$, its $(k,l)$-th entry is given by
\(
\sum_{i=1}^n \tilde{Y}_{ki}\tilde{Y}_{li},
\)
where
\(
\tilde{Y}_{ki} = \left(\sum_{i=1}^n (Y_{ki}-\bar{Y}_k)^2\right)^{-1/2} (Y_{ki} - \bar{Y}_k),
\ \
\bar{Y}_k = \frac{1}{n}\sum_{i=1}^n Y_{ki}.
\)
Consequently, the asymptotic behavior of $\mathbf{T}$ in \citet{dornemann2025ties} differs from that of $\sqrt{n/p}\big(\mathbb{T} - D(\mathbb{T})\big)$ due to the following structural differences between $Y_{ki}$ and $\tilde{Y}_{ki}$:

\begin{center}
\begin{tabular}{|c|c|c|}
\hline 
Property & $Y_{ki}$ & $\tilde{Y}_{ki}$ \\ 
\hline 
Independence across $k$ & Yes & Yes \\ 
\hline 
Independence across $i$ & Yes & No \\ 
\hline 
Row vectors lie on the unit sphere & No & Yes \\ 
\hline 
Expectation & $0$ & $0$ \\ 
\hline 
Variance & depends on $k$ & $n^{-1}$ \\ 
\hline 
$r$-th raw moment & $O(1)$ & $O(n^{-r/2})$ \\ 
\hline 
Expectation of products of $r$ variables & $0$ & $O(n^{-r})$ \\ 
with same $k$ and distinct $i$'s & & \\ 
\hline 
\end{tabular}
\end{center}
Therefore, the results in Corollary~\ref{cor: IID} and Theorem~2.5 of \citet{dornemann2025ties} are not directly comparable. Nevertheless, by following a similar combinatorial approach as in the proof of Theorem~\ref{thm: A}, and exploiting the weak dependence structure (captured by the last row of the table) together with the decay of higher-order moments (second last row), one can recover Theorem~2.5(1) of \citet{dornemann2025ties}.
Details of the modifications required in the proof of Theorem~\ref{thm: A} to obtain Theorem~2.5(1) of \citet{dornemann2025ties} are provided in Remark \ref{rem: heinypf}.
\end{remark}

\noindent We now present an examples where the non-degeneracy condition in Assumption 2.2 of \citet{dornemann2025ties} is violated but (\ref{eqn: g1A}) and Assumption G2-a hold true.  In this case, our results remain valid, whereas the results of \citet{dornemann2025ties} are not applicable. 


%
\begin{example} \label{Example: B}
Let $X_{ki}\sim \mathrm{Cauchy}(0,\sqrt{k})$ for all $i$ and odd $k$. For each even $k$, let $\mathbb{P}(X_{ki}=0)=1-k^{-1/4}$ and $\mathbb{P}(X_{ki}=1)=k^{-1/4}$ for all $i$. Then $\{X_{ki}\}$ are identically distributed across $i$ but not across $k$. Moreover,
\begin{eqnarray}
\limsup_{p\to\infty}\max_{1\le k\le p}\mathbb{P}(X_{k1}=0)
=\lim_{p\to\infty}(1-p^{-1/4})=1,
\end{eqnarray}
so Assumption~2.2 of \citet{dornemann2025ties} is violated. 
In this example, $\mathbb{V}{\rm{ar}}(Y_{k1})=1/3$ for odd $k$ and $k^{-1/4}(1-k^{-1/4})$ for even $k$. Hence the ESD of $\Sigma_p$ converges weakly to a distribution $b$ that assigns mass $0.5$ to each of $0$ and $1/3$. By Corollary~\ref{cor: IID}, the LSD of $\sqrt{n/p}\big(\mathbb{T}-D(\mathbb{T})\big)$ is given by $S_2 b$, where $S_2$ and $b$ are free. 
The first two rows of Table~\ref{tab:Table1aB} agree, showing that the theoretical moments of $S_2 b$ match well with the simulated moments for this example. In contrast, the last two rows of Table~\ref{tab:Table1aB} do not agree, indicating that the results of \cite{dornemann2025ties} are not applicable in this setting.

\begin{table}[h!]
\centering
\caption{Moment comparison for Example~\ref{Example: B} with $n=900$ and $p=30$.}
\label{tab:Table1aB}
\begin{tabular}{|c|c|c|c|c|c|c|}
\hline
 & $r=2$ & $r=3$ & $r=4$ & $r=5$ & $r=6$\\
\hline
Theoretical moment of $6 S_2 b$ & 1 & 0 & 4  & 0 &  20\\
\hline 
Empirical moment of ECDF of $3\sqrt{n/p}\big(\mathbb{T}-D(\mathbb{T})\big)$ & 1.08  & 0.005  & 3.94 & 0.03 & 20.105\\
\hline
Empirical moment of ECDF of $(3/2)\sqrt{n/p}\mathbf{T}$ where & 3.33 & 0.04 & 13.56 & 0.03 & 66.55 \\
 $\mathbf{T}$ is defined after the equation (6) in \citet{dornemann2025ties} &&&&& \\
 \hline
 Standard semi-circle moments & 1 & 0 & 2 & 0 & 5 \\
 \hline
\end{tabular}
\end{table}
\end{example}

\begin{remark} \label{rem: degenerate}
In contrast to \citet{dornemann2025ties}, our centered and scaled Kendall’s correlation matrix does not require normalization, and hence no asymptotic non-degeneracy condition is needed. Degeneracy of observations may produce zero entries but does not imply low rank or a degenerate LSD. The limiting behavior remains non-trivial provided the associated variance structure is non-degenerate.
\end{remark}

\begin{remark} \textbf{State of the art for IID observations}. 
Suppose $\{X_{ki}\}$ satisfy Assumption~1 and are identically distributed across $i$. The following summarizes the implications of our results and those of \citet{dornemann2025ties}.

\begin{enumerate}
\item \textbf{Continuous case.} If $\{X_{ki}\}$ are continuous random variables, then the LSD of 
$\sqrt{n/p}\big(\mathbb{T} - D(\mathbb{T})\big)$ and $\sqrt{n/p}\,\mathbf{T}$ in \citet{dornemann2025ties} coincide, and the common limit is $(2/3)S_2$. See Corollary~\ref{cor: IID1}.

\item \textbf{Assumption 2.2 of \citet{dornemann2025ties} holds, but our conditions fail.} If Assumption~2.2 of \citet{dornemann2025ties} holds, while \eqref{eqn: g1A} or Assumption~G2-a do not hold, then no result is available for the LSD of $\sqrt{n/p}\big(\mathbb{T} - D(\mathbb{T})\big)$. In contrast, the LSD of $\sqrt{n/p}\,\mathbf{T}$ is $(2/3)S_2$, as established in Theorem~2.5(1) of \citet{dornemann2025ties}. Moreover, this result can also be recovered via a minor modification of the proof of Theorem~\ref{thm: A}; see Remark~\ref{rem: heinypf}.

\item \textbf{Assumption 2.2 of \citet{dornemann2025ties} fails, but our conditions hold.} If Assumption~2.2 of \citet{dornemann2025ties} fails, but \eqref{eqn: g1A} and Assumption~G2-a hold, then the LSD of $\sqrt{n/p}\big(\mathbb{T} - D(\mathbb{T})\big)$ is characterized by Corollary~\ref{cor: IID}. In this setting, no corresponding result is available for $\sqrt{n/p}\,\mathbf{T}$ in \citet{dornemann2025ties}. See Example~\ref{Example: B} for further evidence.
\end{enumerate}
\end{remark}

\subsection{Semi-circle LSDs}
\noindent Note that the LSD in Theorem~\ref{thm: A} is model-specific and, in general, cannot be expressed in a unified form unless it reduces to the semicircle law. Even for i.i.d.\ observations, the limit need not be semicircular unless the random variable $a$ in Corollary~\ref{cor: IID} degenerates to a constant. In the following, we restrict our attention to classes of data matrices satisfying Assumption~3 (or Assumption~3A), for which the LSD of $\sqrt{n/p}(\mathbb{T}-D(\mathbb{T}))$ converges to the semicircle law.
\vskip 5pt
\noindent \textbf{Assumption 3}: There exist  $\gamma_1$, $\gamma_2$; a positive integer $n^{(0)}$; constants $C,\delta>0$ so that for all $1 \leq k_1, k_2 \leq p$ and for all $n >n^{(0)}$, we have 
\begin{eqnarray}
(n/p)\  |n^{-2}\sum_{i\in [n]}\tr(\mathbb{G}_{k_1,i})-\gamma_1|&<&Ck_1^{-(1+\delta)},\ \ \ 
(n/p)\  |n^{-3}\sum_{i\in [n]}\tr(\mathbb{G}_{k_1,i}\bG_{k_2,i})-\gamma_2| < C(\min(k_1,k_2))^{-(1+\delta)}. \ \ \
\end{eqnarray}
\begin{remark}
Suppose $\{X_{ki}\}$ satisfy Assumption~1 and are identically distributed across $i$. Then, by Remark~\ref{rem: GG}, Assumption~3 reduces to the following condition: there exist  constants $\tilde{\gamma}_{1}, C>0$ such that
\(
\tilde{\gamma}_1 - C(p/n)k^{-(1+\delta)}
\;<\;
\mathrm{Var}(Y_{k1})
\;<\;
\tilde{\gamma}_1 + C(p/n)k^{-(1+\delta)}
\quad \text{for all } k \ge 1 .
\)
\end{remark}

\begin{remark} Assumption 3 requires that the properly scaled version of $n^{-2}\sum_{i\in [n]}\tr(\mathbb{G}_{k_1,i})$ and $n^{-3}\sum_{i\in [n]}\tr(\mathbb{G}_{k_1,i}\mathbb{G}_{k_2,i})$ can be decomposed into a leading term, independent of $k_1$ and $k_2$, and a remainder term that depends on $k_1, k_2$ but decays sufficiently fast to be summable. This condition does not require identical distributions across both observations and components of observations; rather, it imposes mild aggregate regularity conditions that are sufficient to ensure a semi-circular limiting distribution.  Without such control, the asymptotic behavior reverts to that described in Theorem~\ref{thm: A}, which may yield a non-semicircular limiting distribution, as illustrated in Example~\ref{Example: A}.
\end{remark}

\begin{remark} \label{rem: 3G} Assumption 3 implies Assumptions G1 and G2 with $g_1 = \gamma_1$ and $g_{2\pi} = \gamma_2^{R}$ for all $\pi \in {\rm{NC}}_2(2R)$ and $R\geq 1$, since, with $a_n = \max(\sqrt{np},(n/p))$, we have
\begin{eqnarray}
a_n\Big| n^{-2}p^{-1}\sum_{k,i} \tr(\mathbb{G}_{k,i}) - \gamma_1\Big| &=& a_n\Big| p^{-1}\sum_{k}\big(n^{-2}\sum_{i} \tr(\mathbb{G}_{k,i}) -\gamma_1\big) \Big| \leq a_n p^{-1}\sum_{k}\Big|n^{-2}\sum_{i} \tr(\mathbb{G}_{k,i}) -\gamma_1\Big| \nonumber \\
& \lesssim & a_n (p/n) p^{-1}\sum_{k} k^{-(1+\delta)}  \to 0, \nonumber 
\end{eqnarray}
\begin{eqnarray}
&& \bigg|(n^{-R}p^{-1})\sum_{k_1,\ldots,k_{2R}}
\Big(
\prod_{(b,s)\in \pi}
\Big(
p^{-1}n^{-2}\sum_{i_b}\tr(\mathbb{G}_{k_b,i_b}\mathbb{G}_{k_{\sigma_{2R}(b)},i_b})
\,
\delta_{k_b,k_{\sigma_{2R}(s)}}
\delta_{k_s,k_{\sigma_{2R}(b)}}
\delta_{i_b,i_s}
\Big)
\Big) - \gamma_2^{R}\bigg|\nonumber \\
& = &   (n^{-R}p^{-1})\sum_{k_1,\ldots,k_{2R}}
\left(
\prod_{(b,s)\in \pi}
\Big(
p^{-1}n^{-2}\sum_{i_b}\tr(\mathbb{G}_{k_b,i_b}\mathbb{G}_{k_{\sigma_{2R}(b)},i_b}) - (n/p)\gamma_2 \Big)
\,
\delta_{k_b,k_{\sigma_{2R}(s)}}
\delta_{k_s,k_{\sigma_{2R}(b)}}
\delta_{i_b,i_s}
\right) \nonumber \\
& \lesssim & (n^{-R}p^{-1})\sum_{k_1,\ldots,k_{2R}} \prod_{(b,s)\in \pi}\min(k_b,k_{\sigma_{2R}(b)})^{-(1+\delta)}\delta_{k_b,k_{\sigma_{2R}(s)}}
\delta_{k_s,k_{\sigma_{2R}(b)}} \lesssim  (n^{-R}p^{-1})\sum_{k_1,\ldots,k_{2R}}  (\min\{k_j:\ 1 \leq j \leq r\})^{-(1+\delta)} \nonumber\\
& \lesssim & n^{-R}p^{R-1} \to 0. \nonumber 
\end{eqnarray}
\end{remark}
\noindent Assumption 3 can be relaxed further as follows  when $\{X_{ki}:\ 1 \leq i \leq n\}$ are identically distributed across $1 \leq k \leq p$.
\vskip 10pt
\noindent \textbf{Assumption 3A}: $\{X_{ki}:\ 1 \leq i \leq n\}$ are identically distributed across $1 \leq k \leq p$ and there are  numbers $\gamma_1$, $\gamma_2$ such that
\begin{eqnarray}
    n^{-(r+1)}\sum_{i\in [n]}\tr(\mathbb{G}_{k,i})^{r} \to \gamma_r\ \ \text{for all}\ r=1,2. \nonumber 
\end{eqnarray}
\begin{remark}
Suppose $\{X_{ki}\}$ are i.i.d.\ random variables and their common distribution does not depend on $n$. Then Assumption~3A is automatically satisfied.
\end{remark}
\noindent  We present examples in Section \ref{sec: example} with non-identical observations, to illustrate that Assumptions 1–3 (or Assumption 3A) are often satisfied.
The following theorem states about the LSD of appropriately centered and scaled version of $\bT - {\rm{D}}(\bT)$ under Assumption 3. Its proof is immediate from Theorem \ref{thm: A} and Remark \ref{rem: 3G}.  
\begin{theorem} \label{thm: 1}
Suppose Assumptions 1-3 hold, (or 3A holds instead of 3) and $p := p(n) \to \infty$, $p/n \to 0$ as $n \to \infty$. Then the  ESD of $\sqrt{\frac{n}{p}} (\mathbb{T}-D(\mathbb{T}))$ converges weakly almost surely to  $2S_{2\sqrt{\gamma_2}}$. 
\end{theorem}

\subsection{Examples}  \label{sec: example}
\noindent In this section, we present three examples that are not addressed in the existing literature. In these examples, the $X_{ki}$’s are identically distributed across $k$ but not across $i$. 
Examples \ref{example: 1} considers heavy-tailed discrete observations, while Example \ref{example: 4} involves a mixture of both heavy-tailed  continuous and discrete observations. Across all these examples, Assumptions 1, 2 and 3A hold, ensuring that the conclusion of Theorem \ref{thm: 1} applies.

\begin{example} \label{example: 1}
Let $X_{ki}$ be discrete random variables that are independent across $1 \leq i \leq n$ and independently and identically distributed across $1 \leq k \leq p$. If $i$ is even, the probability mass function is 
$\bP(X_{ik}=z)=\frac{45}{\pi^{4}z^4}$ for all $z = \pm 1, \pm 2,\ldots$.
For odd $i$, $X_{ki}$ is a scaled Rademacher random variable i.e. $\bP(X_{ik} = \pm 2) = 0.5$. Also assume that $n,p \to \infty$ and $p/n \to 0$. It can be easily seen that Assumption 1 is satisfied. We will now argue that this model satisfies Assumptions 2 and 3A. \vskip 3pt
\noindent 
\textit{Verification of Assumption 2}: Suppose both $i$ and $j$ are of same parity. Then
\begin{eqnarray}
    \bP(X_{kj} < X_{ki})&=& \sum_{x_{ki}}\bP(X_{kj}<X_{ki}|X_{ki}=x_{ki})\bP(X_{ki}=x_{ki}) =\sum_{x_{ki}}\bP(X_{kj}<x_{ki})\bP(X_{ki}=x_{ki})\nonumber\\
    &=&\sum_{x_{ki}}\bP(X_{kj} > -x_{ki})\bP(X_{ki}=-x_{ki})\label{eq:ex1.1} =\sum_{y_{ki}}\bP(X_{kj}>y_{ki})\bP(X_{ki}=y_{ki}) =\bP(X_{kj}>X_{ki}).
    \end{eqnarray}
Note that (\ref{eq:ex1.1}) uses the symmetric property of the distribution under consideration.  Next, suppose $i$ is even and $j$ is odd. Then 
\begin{eqnarray}
    \bP(X_{kj}<X_{ki})&=&\sum_{x=1}^{\infty}\bP(X_{kj}<x)\bP(X_{ki}=x) = \frac{1}{2}\bigg(2\frac{45}{\pi^4}+\frac{45}{16\pi^4 }\bigg)+\bigg(
    \frac{1}{2} - \frac{45}{\pi^2}-\frac{45}{16\pi^4 }\bigg) =\frac{1}{2}-\frac{1}{2}\frac{45}{16\pi^4}, \nonumber \\
    \bP(X_{kj}>X_{ki})&=&1-\bP(X_{kj}\leq X_{ki})=1-\sum_{x=-1}^{\infty}\bP(X_{kj} \leq x)\bP(X_{ki}=x) = 1-\frac{1}{2}\bigg(\frac{45}{16\pi^4} +  2\frac{45}{\pi^4}\bigg)- \bigg(\frac{1}{2} - \frac{45}{\pi^4}\bigg) = \frac{1}{2}-\frac{1}{2}\frac{45}{16\pi^4}. \nonumber 
\end{eqnarray}
The case where $i$ is odd and $j$ is even can be treated similarly.
This completes the proof that Assumption 2 holds.\vskip 3pt
\noindent \textit{Verification of Assumption 3A}:  To this end, $\bG_{k,i}$ needs to be evaluated. Since $X_{ki}$ are identically distributed across $k$, $\bG_{k,i}$ are same across $k$ for all $1 \leq k \leq p$. Denote $\bG_{k,i}$ by $\bG^{(1,1)}$ and $\bG^{(1,2)}$ for even and odd $i$ respectively. Numerical values of entries of these matrices can be computed using the R script available at:  
\url{https://github.com/Rss-1313/Simulations/blob/main/Discrete_Discrete_calculation.R}.
Final results are as follows: 
\begin{eqnarray}
\bG^{(1,1)}(j_1,j_2) = \begin{cases} 0.03275,\ \ \text{if $j_1$, $j_2$ both are odd}, \\
0.4566,\ \ \text{if $j_1$, $j_2$ are of opposite parity}, \\
0.2675,\ \ \text{if $j_1$, $j_2$ both are even}, 
\end{cases}  \ \ \ \bG^{(1,2)}(j_1,j_2) = \begin{cases} 0.25,\ \ \text{if $j_1$, $j_2$ both are odd}, \\
0.4764,\ \ \text{if $j_1$, $j_2$ are of opposite parity}, \\
0.9078,\ \ \text{if $j_1$, $j_2$ both are even}.
\end{cases}   \nonumber 
\end{eqnarray}
Clearly,
\begin{eqnarray}
n^{-2}\sum_{i=1}^n \tr(\bG_{k,i}) &\to & (0.03275+0.2675+0.25+0.9078)/4 = 0.3645 \ \nonumber \\
 \text{and}\ \ n^{-3}\sum_{i=1}^n \tr(\bG_{k,i}^2) &\to & ((0.03275)^2+ 2(0.4566)^2+(0.2675)^2+ (0.25)^2 + 2(0.4764)^2+(0.9078)^2)/8 = 0.242. \nonumber 
\end{eqnarray}
Therefore, the model in Example \ref{example: 1} satisfies Assumption 3A with $\gamma_1 = 0.3645$ and $\gamma_2 = 0.242$. \vskip 3pt

Consequently, since $2\sqrt{0.242} = 0.984$, Theorem~\ref{thm: 1} implies that the ESD of $\sqrt{n/p}\,(\bT - \mathbb{D}(\bT))$ converges weakly, almost surely, to $2S_{0.984}$. Equivalently, the ESD of $\sqrt{n/p}\,(\bT - \mathbb{D}(\bT))/0.984$ converges weakly, almost surely, to $S_{2}$.

The simulation results reported in the third row of Table~\ref{tab:Table1a} show that the sample moments of the ESD of $\sqrt{n/p}\,(\bT - \mathbb{D}(\bT))/0.984$ closely match those of $S_{2}$, as given in the second row of the table. In contrast, the discrepancy observed in the fourth row of Table~\ref{tab:Table1a} indicates that the results of \citet{dornemann2025ties} do not apply to this example.

\begin{table}[h!]
  \begin{center}
    \caption{Moment comparison for Examples \ref{example: 1} - \ref{example: 4}  with $n=4900$ and $p=70$. \\ Let $\omega_1$ and $\omega_2$ be respectively the ESD of  $\sqrt{n/p}(\mathbb{T-D(T)})/2\sqrt{\gamma_2}$ and $\mathcal{W}_2 = (3/2)\sqrt{n/p}\mathbf{T}$ considered in Theorem 2.5 (1) of \citet{dornemann2025ties}}
    \label{tab:Table1a}
    \begin{tabular}{|c|c|c|c|c|c|c|c|} 
         \hline 
       The $r$-th order moment  & $r=2$ & $r=3$ & $r=4$ & $r=5$ & $r=6$ & $r=7$ & $r=8$\\
        
      \hline
      $S_{2}$  & 1& 0 & 2 & 0 & 5 & 0 & 14\\
      \hline
      Example \ref{example: 1}: $\omega_1$   & 0.97  & 7 $\times 10^{-9}$   & 1.96 & 13 $\times 10^{-7}$ & 5.07 & 0.0001 & 13.93
      \\
      \hline
      Example \ref{example: 1}: $\omega_2$   & 3.82  & 0.0004  & 40.44 & 0.0021 & 312.22 & 0.02 & 3670.06\\
      \hline
       Example \ref{example: 2}: $\omega_1$  & 1.02  & 12 $\times 10^{-10}$   & 2.01 & 8 $\times 10^{-6}$ & 4.96 & 0.0003 & 13.89\\
      \hline
       Example \ref{example: 2}: $\omega_2$  & 1.37  & 0.0006  & 3.64 & 0.0013 & 12.19 & 0.004 & 45.88\\
      \hline
       Example \ref{example: 4}: $\omega_1$  & 0.98  & 8 $\times 10^{-10}$   & 1.97 & 11 $\times 10^{-6}$ & 5.03 & 0.0001 & 14.06\\
      \hline
       Example \ref{example: 4}: $\omega_2$  & 2.93  & 0.0003  & 12.76 & 0.0009 & 71.19 & 0.007 & 224.77\\
      \hline
    \end{tabular}
  \end{center}
\end{table}
\end{example}

\begin{example} \label{example: 2}
One of the most commonly used distributions in statistics is the normal distribution. This example considers observations arising from normal distributions. Suppose $X_{ki}$ follows $N(5,\sigma_i^2)$ where $\sigma_i= i$. Note that observations are   independent across $i$ and independently and identically distributed across $k$. It is easy to see that Assumption 1 holds. Assumption 2 holds because of the symmetric nature of the normal distribution.  Next we shall discuss Assumption 3A for this model.  It is easy to see that, for all $i,j_1,j_2$, we have 
$$\text{$\bG_{k,i}(j_1,j_2)=\frac{2}{\pi}\sin^{-1}{\rho_{i,j_1,j_2}}$ where $\rho_{i,j_1,j_2} =\frac{\sigma^{2}_i}{(\sigma_i^2+\sigma^2_{j_1})^{0.5}(\sigma^2_i+\sigma^2_{j_2})^{0.5}}$.}$$ 
Since it is analytically difficult to evaluate $\tr(\bG_{k,i}^r)$ for $r=1,2$; we resorted to numerical analysis. Using this analysis, an estimate of $\gamma_1$ and $\gamma_2$ are obtained, say $\hat{\gamma_1}$ and  $\hat{\gamma_2}$. The R program to evaluate this is given in the 
link \url{https://github.com/Rss-1313/LSDp-n-to-0/blob/main/Gkinormal900.R}. A value of  $\hat{\gamma_1}$ and  $\hat{\gamma_2}$ are thus obtained $0.3783$ and $0.1622466$ respectively. 

\vskip 3pt

Consequently, since $2\sqrt{0.1622466} = 0.81$, Theorem~\ref{thm: 1} implies that the ESD of $\sqrt{n/p}\,(\bT - \mathbb{D}(\bT))$ converges weakly, almost surely, to $2S_{0.81}$. Equivalently, the ESD of $\sqrt{n/p}\,(\bT - \mathbb{D}(\bT))/0.81$ converges weakly, almost surely, to $S_{2}$.

The simulation results reported in the fifth row of Table~\ref{tab:Table1a} show that the sample moments of the ESD of $\sqrt{n/p}\,(\bT - \mathbb{D}(\bT))/0.81$ closely match those of $S_{2}$, as given in the second row of the table. In contrast, the discrepancy observed in the sixth row of Table~\ref{tab:Table1a} indicates that the results of \citet{dornemann2025ties} do not apply to this example.

%
\end{example}

\begin{example} \label{example: 4}
This section focuses on observations arising from both discrete and continuous distributions. For the sake of this example, let $X_{ki} \sim t(1)$ ( i.e $t$-distribution with $1$ degrees of freedom) for even $i$  and $\bP(X_{ki}= \pm 1)=0.5$ for odd $i$. It can be easily seen that observations are independent across both $k$ and $i$. It is therefore evident that this setup satisfies Assumption 1. Note that observations are not identically distributed across $i$. Now we shall discuss whether this model satisfies Assumptions 2 and 3A. 
\vskip 3pt
\noindent \textit{Verification of Assumption 2}: $\bP(X_{ki}<X_{kj}) =\bP(X_{ki}>X_{kj})$  holds trivially for even $i$ and $j$ since $X_{ki}$ and $X_{kj}$ are independently and identically distributed and continuous random variables. Next, suppose $i$ and $j$ are both odd, then 
\begin{eqnarray}
\bP(X_{ki}<X_{kj}) &=& 1-\bP(X_{kj} \leq X_{ki}) = 1- \sum_{x_{ki}=-1,1}\bP(X_{kj} \leq x_{ki})\bP(X_{ki} = x_{ki}) = 1- \frac{1}{2}\frac{1}{2}-\frac{1}{2} = \frac{1}{4}, \nonumber \\
\bP(X_{ki} > X_{kj}) &=& \bP(X_{kj} < X_{ki}) =   \sum_{x_{ki}=-1,1}\bP(X_{kj} < x_{ki})\bP(X_{ki} = x_{ki}) =  \frac{1}{2}\frac{1}{2} = \frac{1}{4}. \nonumber 
\end{eqnarray}
Now, suppose $i$ is odd and $j$ is even. Then 
\begin{eqnarray}
\bP(X_{kj} >X_{ki}) &=& 1-\bP(X_{kj} \leq X_{ki}) = 1- \sum_{x_{ki}=-1,1}\bP(X_{kj} \leq x_{ki})\bP(X_{ki} = x_{ki}) = 1-  \frac{1}{4}\frac{1}{2}-\frac{3}{4}\frac{1}{2} = \frac{1}{2}, \nonumber \\
\bP(X_{kj} < X_{ki}) &=&   \sum_{x_{ki}=-1,1}\bP(X_{kj} < x_{ki})\bP(X_{ki} = x_{ki}) =  \frac{1}{4}\frac{1}{2}+\frac{3}{4}\frac{1}{2}=\frac{1}{2}. \nonumber 
\end{eqnarray}
The case where $i$ is even and $j$ is odd can be handled similarly.
This completes the proof that Assumption 2 holds for the model in Example \ref{example: 4}.
\vskip 3pt
\noindent \textit{Verification of Assumption 3A}: To verify Assumption 3A, it is essential to establish convergence of  $n^{-s-1}\sum_{i}\tr(\bG_{k,i}^s)$ for  $s=1,2$. The algorithm for establishing the said convergence is the same as the one followed for Example \ref{example: 1} and hence the details are omitted here.  The link of R script for this computation is \url{https://github.com/Rss-1313/LSDp-n-to-0/blob/main/discontevalpn0.R} and it proves that Assumption 3A holds for Example \ref{example: 4} with $\gamma_1 = 0.3343$ and $\gamma_2 = 0.1115451$. \vskip 3pt

Consequently, since $2\sqrt{\gamma_2} = 0.668$, Theorem~\ref{thm: 1} implies that the ESD of $\sqrt{n/p}\,(\bT - \mathbb{D}(\bT))$ converges weakly, almost surely, to $2S_{0.668}$. Equivalently, the ESD of $\sqrt{n/p}\,(\bT - \mathbb{D}(\bT))/0.668$ converges weakly, almost surely, to $S_{2}$.

The simulation results reported in the seventh row of Table~\ref{tab:Table1a} show that the sample moments of the ESD of $\sqrt{n/p}\,(\bT - \mathbb{D}(\bT))/0.668$ closely match those of $S_{2}$, as given in the second row of the table. In contrast, the discrepancy observed in the last row of Table~\ref{tab:Table1a} indicates that the results of \citet{dornemann2025ties} do not apply to this example.

\end{example}

\begin{remark}
While \citet{kendep2023} establishes the LSD of Kendall’s correlation matrices under cross-component dependence in the i.i.d.\ continuous setting (recovering \citet{ECP} as a special case), and \citet{dornemann2025ties} treats the i.i.d.\ discrete case under independence, our work further extends \citet{dornemann2025ties} to non-identically distributed observations and certain discrete/mixed i.i.d.\ settings beyond their assumptions. Both \citet{dornemann2025ties} and this paper rely crucially on independence across components. Extending our results to dependent components would require substantially different techniques, particularly in the presence of distributional heterogeneity, and is therefore left for future research.
\end{remark}

\begin{remark}
The asymptotic regimes $p/n \to 0$, $p/n \to \theta \in (0,\infty)$, and $p/n \to \infty$ typically lead to fundamentally different limiting spectral behaviors in random matrix theory and cannot, in general, be treated within a unified framework. The present paper focuses on the regime $p/n \to 0$; the proportional and ultra-high dimensional cases require substantially different techniques, particularly in the presence of heterogeneity, and are left for future investigation.
\end{remark}

\subsection{Data-driven ad-hoc verification of assumptions via clustering} \label{sec: adhoc}

In general, when the entries of the data matrix are all drawn from distinct distributions, it is not feasible to directly verify the structural assumptions imposed in our theoretical framework. However, a practically relevant and tractable setting arises when the entries can be grouped into some clusters, where observations within each cluster share the same underlying distribution. This allows us to reduce the complexity of the problem by working with a smaller number of distributional components.

More precisely, suppose that the index set $\{(k,i): 1 \leq k \leq p,\, 1 \leq i \leq n\}$ can be partitioned into $L_n$ clusters $\Omega_1, \Omega_2, \ldots, \Omega_{L_n}$ such that all entries in a given cluster are identically distributed. We allow the cluster sizes to vary, but assume that each cluster is sufficiently large to permit reliable estimation of the corresponding distribution function, while the total number of clusters $L_n$ is growing but small relative to $np$.

In practice, the true clustering structure is unknown. Therefore, we adopt a data-driven approach in which we first apply a suitable clustering algorithm to partition the indices into estimated clusters $\hat{\Omega}_1, \hat{\Omega}_2, \ldots, \hat{\Omega}_{L_n}$. These estimated clusters serve as proxies for the latent groups of identically distributed entries.

For each estimated cluster $\hat{\Omega}_u$, we define the empirical distribution function
\begin{eqnarray}
\hat{F}_u (t) = \frac{1}{|\hat{\Omega}_u|} \sum_{(k,i) \in \hat{\Omega}_u} \bI(X_{ki}\leq t), \nonumber 
\end{eqnarray}
where $\bI(\cdot)$ is the indicator function and which provides an estimate of the common distribution function within that cluster. We note that this estimator is subject to two sources of error: (i) misclassification error arising from inaccuracies in the clustering step, and (ii) sampling error inherent in empirical distribution estimation.

Using these empirical distribution functions, we propose to assess the validity of Assumption~2 through a symmetry condition across clusters. Specifically, for any pair of clusters $u \neq v$, we consider the quantities
\begin{eqnarray}
A_{u,v} = \frac{1}{|\hat{\Omega}_v|} \sum_{(k,i) \in \hat{\Omega}_v}  
\big(\hat{F}_{u}(X_{ki}) + \bar{\hat{F}}_{u}(X_{ki}) -1\big), \ \ \ A_{v,u} = \frac{1}{|\hat{\Omega}_u|} \sum_{(k,i) \in \hat{\Omega}_u}  
\big(\hat{F}_{v}(X_{ki}) + \bar{\hat{F}}_{v}(X_{ki}) -1\big). \nonumber
\end{eqnarray}
If $A_{u,v} \approx A_{v,u}$ for all pairs $u \neq v$, this provides empirical evidence supporting Assumption~2.

Next, we turn to the estimation of the covariance structure $\mathbb{G}_{k,i}$. For each $(k,i)$, let $\tilde{\Omega}_{ki}$ denote the estimated cluster associated with that entry, that is, $\tilde{\Omega}_{ki} = \hat{\Omega}_u$ if $(k,i) \in {\Omega}_u$. We then define the estimator
\begin{eqnarray}
\hat{\mathbb{G}}_{k,i}(j_1,j_2) 
&=& \frac{1}{|\tilde{\Omega}_{ki}|} \sum_{(k^\prime,i^\prime) \in \tilde{\Omega}_{ki}} 
\big(\hat{F}_{\tilde{\Omega}_{k j_1}}(X_{k^\prime i^\prime}) + \bar{\hat{F}}_{\tilde{\Omega}_{k j_1}}(X_{k^\prime i^\prime}) -1\big) 
\big(\hat{F}_{\tilde{\Omega}_{k j_2}}(X_{k^\prime i^\prime}) + \bar{\hat{F}}_{\tilde{\Omega}_{k j_2}}(X_{k^\prime i^\prime}) -1\big). \nonumber
\end{eqnarray}
This estimator is obtained by pooling information within the estimated cluster $\tilde{\Omega}_{ki}$ and substituting the unknown distribution functions with their empirical counterparts.

Once $\hat{\mathbb{G}}_{k,i}$ is constructed, we can empirically evaluate Assumptions G1, G2, and 3 (or 3A) by replacing $\mathbb{G}_{k,i}$ with $\hat{\mathbb{G}}_{k,i}$ in the corresponding expressions. In this way, the theoretical conditions can be translated into computable quantities based solely on the observed data.

We emphasize that the above procedure is heuristic in nature and is intended primarily as a practical guideline for model assessment. A rigorous justification would require precise control of both clustering errors and the estimation errors of the empirical distribution functions, together with a quantitative analysis of their impact on the verification of the assumptions.

More importantly, the present work should be viewed as a first step towards developing a theoretical framework for handling non-identically distributed observations in moderate high-dimensional settings. In the absence of results such as Theorems~\ref{thm: A} and \ref{thm: 1}, it is difficult to design principled and theoretically justified procedures for testing independence under heterogeneous distributions.
This is analogous to the setting of continuous i.i.d.\ observations, where limit theorems for the spectral distribution of Kendall's correlation matrix, established in \citet{ECP}, have been instrumental in enabling the development of rigorous independence tests based on linear spectral statistics in the proportional high-dimensional regime; see \citet{li2021central}.


In contrast, the non-identically distributed setting presents significantly greater challenges, and the theoretical results developed in this paper provide a necessary foundation for future work in this direction. The independence testing framework presented in the next section is therefore exploratory, and is included mainly to illustrate that ignoring heterogeneity in distributions may lead to spurious detection of dependence. Developing a fully rigorous testing methodology that accounts for these complexities remains an important direction for future research.

\section{Test of Independence} \label{sec: application}

Existing tests of independence for heavy-tailed data typically assume that the observations admit a density with respect to the Lebesgue measure on $\mathbb{R}$ and that the columns of a $p \times n$ data matrix are identically distributed, with the dimensional ratio satisfying $p/n \to \theta \in (0, \infty)$. Recently, \citet{dornemann2025ties} studied the LSD of Kendall’s correlation matrix under independent and identically distributed observations, allowing for both discrete and continuous data in the regime $p/n \to \theta \in [0,\infty)$. Although not explicitly discussed there, their results can also be used to construct tests of independence within the i.i.d.\ framework.

In contrast, Theorems~\ref{thm: A} and \ref{thm: 1} provide a framework for testing row independence that remains applicable even when the observations are discrete and the column distributions are not identical, under the regime $p/n \to 0$. It is important to note, however, that the limiting spectral distribution of $\sqrt{n/p}(\bT - \mathbb{D}(\bT))$ in Theorems~\ref{thm: A} and \ref{thm: 1} is, in general, not distribution-free. Consequently, the testing procedures described below are also not distribution-free.

As illustrated below (see Table~\ref{tab:Table6}), failing to account for non-identically distributed observations may lead to spurious detection of dependence. The discussion in this section is intentionally exploratory and based on an ad-hoc procedure, aimed at highlighting the importance of properly addressing heterogeneity in distributions when testing for independence. A rigorous treatment, including the construction of appropriate test statistics (for example, based on linear spectral statistics) and a systematic theoretical analysis, requires substantial further development and is beyond the scope of the present work.

\vskip 5pt
\noindent
Let $\tilde{\bZ} = \{Z_{ki}: 1 \leq k \leq p,\; 1 \leq i \leq n\}$ and $\tilde{\bX} = \{X_{ki}: 1 \leq k \leq p,\; 1 \leq i \leq n\}$ be two independent and identically distributed data matrices, each with independent entries satisfying Assumptions 1, 2, G1 and G2 (or Assumption 3 or 3A). Consider a $p \times p$ matrix $\bA$ whose $(i,j)$-th entry is $\alpha^{\,i-j+1}$ for $i > j$ and for some $\alpha > 0$, and is zero otherwise. We observe the data matrix
\[
\bZ = \tilde{\bX} + \bA \tilde{\bZ}.
\]
Our goal is to test the null hypothesis $H_0: \bA = 0$ against the alternative $H_1: \bA \neq 0$.

\vskip 5pt
\noindent
\textbf{Graphical testing procedure.} The proposed graphical test proceeds as follows:
\begin{enumerate}
\item Compute the Kendall's correlation matrix $\bT_{\bZ}$ from the observed data matrix $\bZ$.
\item Simulate an independent reference matrix $\tilde{\bX}$.
\item Compute the Kendall's correlation matrix $\bT_{\tilde{\bX}}$.
\item Plot the distribution function of ESDs of 
\(
\sqrt{\frac{n}{p}}\big(\bT_{\bZ} - \mathbb{D}(\bT_{\bZ})\big)
\ \ \text{and} \ \ 
\sqrt{\frac{n}{p}}\big(\bT_{\tilde{\bX}} - \mathbb{D}(\bT_{\tilde{\bX}})\big)
\)
on the same graph.
\item If the two ESDs are close, then do not reject $H_0$; otherwise, reject $H_0$.
\end{enumerate}

\noindent Figure~\ref{fig: 5}  illustrates the performance of this procedure. The first row corresponds to the null hypothesis, where the empirical distributions coincide, while the remaining rows correspond to alternatives, where clear deviations are observed.

\noindent
\textbf{Empirical size.} To evaluate the empirical size of the test, we proceed as follows:
\begin{enumerate}
\item Simulate $\bZ$ under $H_0$ (i.e., $\bA = 0$) and compute $\bT_{\bZ}$.
\item Simulate $\tilde{\bX}$ and compute $\bT_{\tilde{\bX}}$.
\item Compute the Kolmogorov distance between the ESDs of 
\(
\sqrt{\frac{n}{p}}\big(\bT_{\bZ} - \mathbb{D}(\bT_{\bZ})\big)
\quad \text{and} \quad
\sqrt{\frac{n}{p}}\big(\bT_{\tilde{\bX}} - \mathbb{D}(\bT_{\tilde{\bX}})\big).
\)
\item Repeat steps 1--3 independently 2000 times and record the distances.
\item Define the cutoff as the $95^{\text{th}}$ percentile of these distances.
\item Repeat steps 1--3 another 1000 times and compute the proportion of distances exceeding the cutoff. This gives the empirical size.
\end{enumerate}

\noindent
\textbf{Empirical power.} To evaluate the empirical power:
\begin{enumerate}
\item Simulate $\bZ$ under $H_1$ (i.e., $\bA \neq 0$) and compute $\bT_{\bZ}$.
\item Simulate $\tilde{\bX}$ and compute $\bT_{\tilde{\bX}}$.
\item Compute the Kolmogorov distance between the corresponding ESDs.
\item Repeat the above steps 2000 times.
\item Compute the proportion of distances exceeding the cutoff obtained from the size analysis. This gives the empirical power.
\end{enumerate}

\noindent Table~\ref{tab:Table6}  demonstrates that the proposed test exhibits good empirical size control and satisfactory power.

\vskip 5pt
\noindent
\textbf{Practical considerations.} In practice, the distribution of $\tilde{\bX}$ is unknown and cannot be directly simulated. To address this, one may first apply a clustering algorithm to group entries of the data matrix into clusters with approximately identical distributions. The distribution function within each cluster can then be estimated empirically, and a synthetic reference matrix can be generated using independent entries drawn from these estimated distributions. Under $H_0$, such a surrogate matrix is expected to mimic the behavior of $\tilde{\bX}$, up to errors arising from misclassification and distributional estimation. Quantifying this approximation error remains an important open problem.

\noindent
\textbf{Limitations.} The performance of the proposed test depends on the signal strength. In particular, when the variability of $\tilde{\bX}$ dominates that of $\bA\tilde{\bZ}$, the resulting Kendall's correlation matrix of $\bZ$ may remain close to that of $\tilde{\bX}$, leading to low power. A precise characterization of the minimal detectable signal strength is beyond the scope of this work and will be addressed in future research.

\noindent
\textbf{Conclusion.} The primary objective of this section is to highlight that ignoring heterogeneity in the distributions of the data matrix can lead to misleading conclusions, in particular, spurious detection of dependence. Properly accounting for non-identically distributed observations is therefore crucial in high-dimensional independence testing. 

For comparison, one may also construct a similar independence test based on Theorem 2.5 (1) of \citet{dornemann2025ties}, and we report its empirical performance in Table~\ref{tab:Table6}. The results indicate that this alternative approach tends to reject the null hypothesis more frequently even when it is true, suggesting potential size distortion under heterogeneous distributions. In addition, its empirical power is generally lower than that of the proposed method.

\begin{table}[h!]
\begin{center}
\caption{Empirical size and power of the proposed test and the test based on Theorem 2.5 (1) of \citep{dornemann2025ties}, for different values of $(n,p)$ and $\alpha$, based on 2000  replications. For each value of $\alpha$, two columns are reported: the first corresponds to the proposed test, and the second corresponds to the test derived from \citep{dornemann2025ties}. The observations are independent and distributed as follows: for even $i$, $X_{ki}$ has a symmetric heavy-tailed distribution given by $\mathbb{P}(X_{ki} = z) = \mathbb{P}(X_{ki} = -z) = \frac{45}{\pi^4 z^4}$ for $z = \pm1, \pm2, \ldots$; for odd $i$, $\mathbb{P}(X_{ki} = \pm2) = 0.3$ and $\mathbb{P}(X_{ki} = \pm1) = 0.2$.}
\label{tab:Table6}
\small
\begin{tabular}{|c|cc|cc|cc|}
\hline
 & \multicolumn{2}{c|}{$\alpha=0$} & \multicolumn{2}{c|}{$\alpha=1$} & \multicolumn{2}{c|}{$\alpha=2$} \\
\hline
$(n,p)$ & Ours & \citep{dornemann2025ties} & Ours & \citep{dornemann2025ties} & Ours & \citep{dornemann2025ties} \\
\hline
$(49, 7)$   & 0.06  &  0.56 & 0.808 & 0.764 & 0.862 & 0.72 \\
\hline
$(100, 10)$ & 0.058 & 0.72 & 0.992 & 0.81 & 0.994 & 0.78 \\
\hline
$(900, 30)$ & 0.054 &  0.77 & 1 & 0.812    &  1 &  0.88         \\
\hline
\end{tabular}
\end{center}
\end{table}

\begin{figure}[ht] \vspace{0 cm}
\includegraphics[height=60mm, width = 50 mm,trim={0 0 0 20mm},clip]{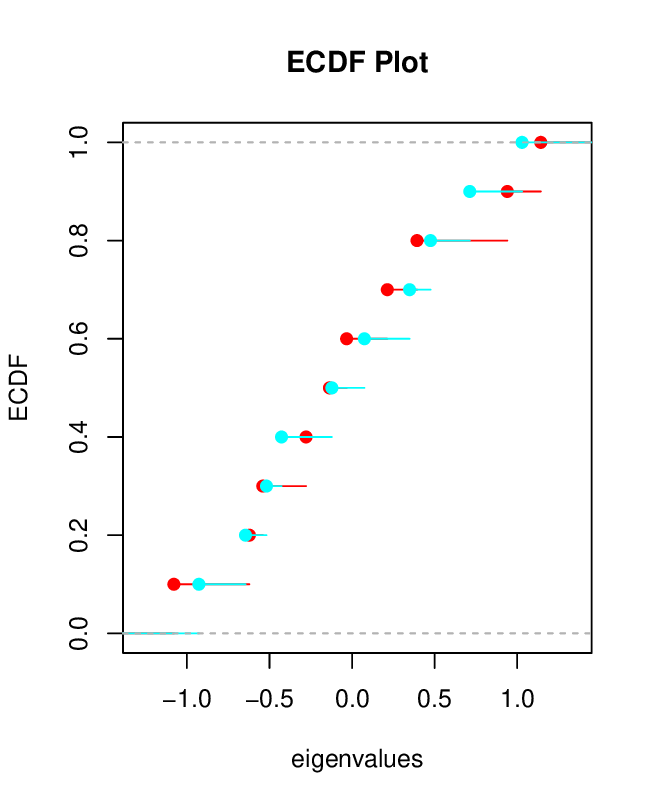} \hspace{0cm}
\includegraphics[height=60mm, width = 50 mm,trim={0 0 0 20mm},clip]{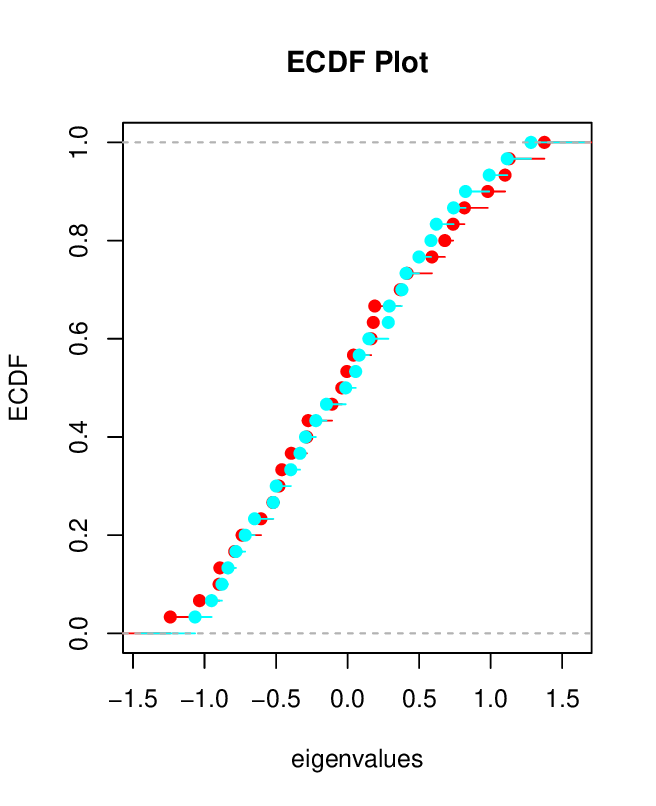}\hspace{0cm}
\includegraphics[height=60mm, width = 50 mm,trim={0 0 0 20mm},clip]{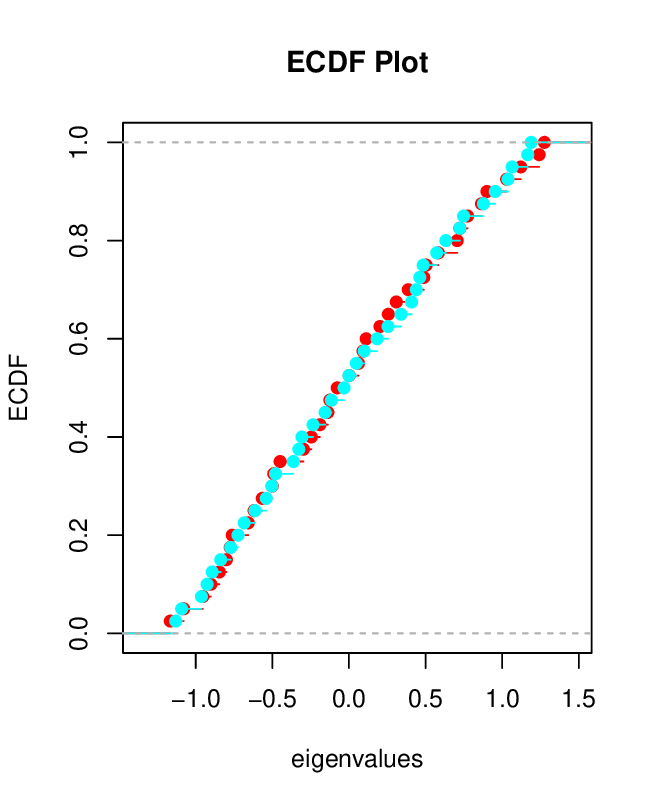} \vspace{0cm}
\includegraphics[height=60mm, width = 50 mm,trim={0 0 0 20mm},clip]{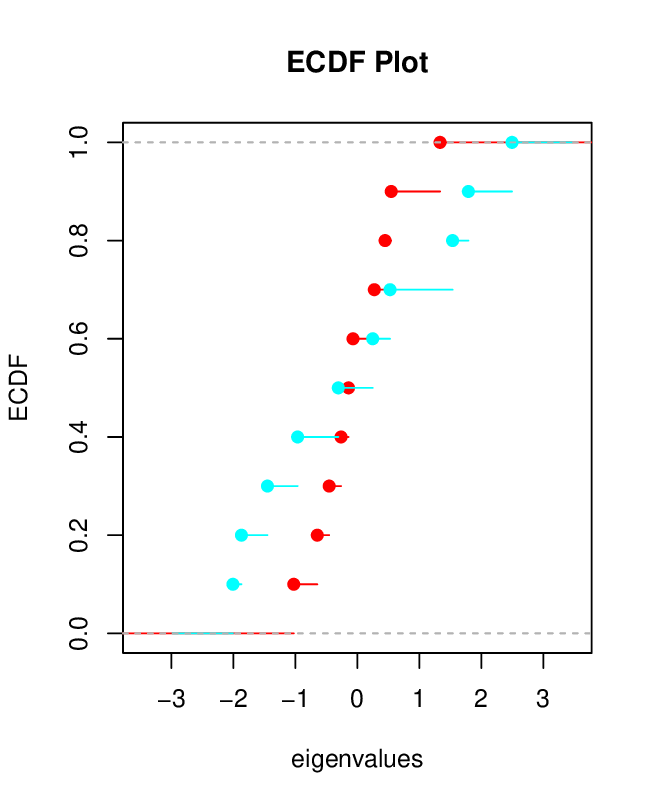} \hspace{0cm}
\includegraphics[height=60mm, width = 50 mm,trim={0 0 0 20mm},clip]{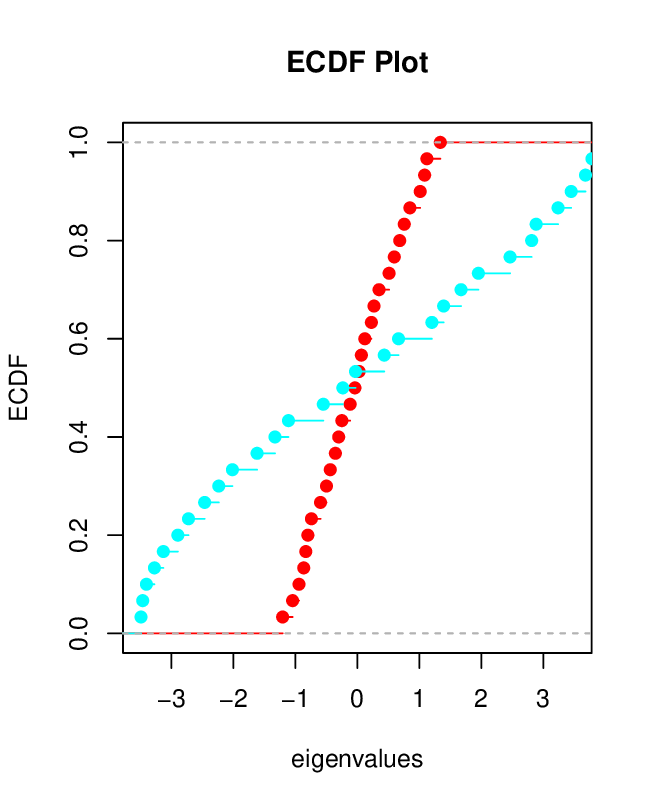} \hspace{0cm}
\includegraphics[height=60mm, width = 50 mm,trim={0 0 0 20mm},clip]{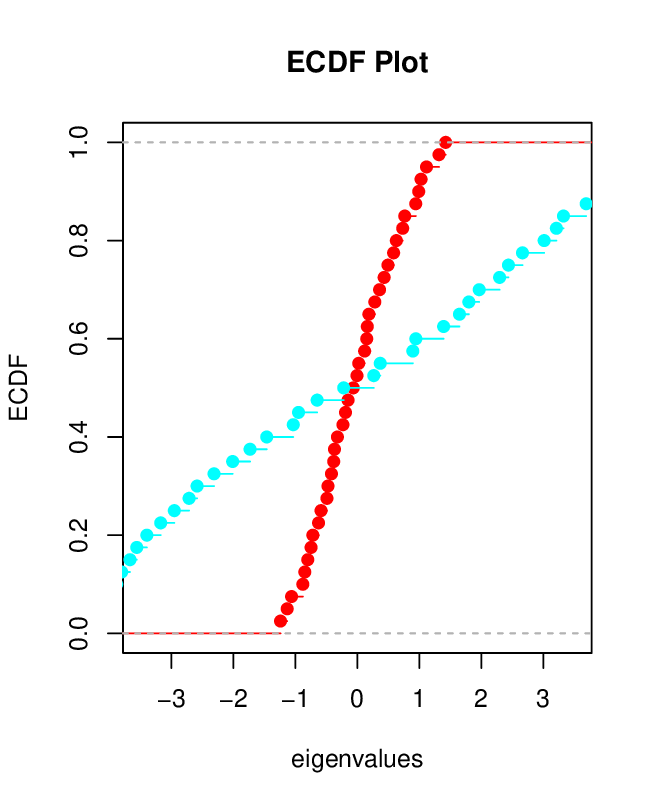} \vspace{0cm}
\includegraphics[height=60mm, width = 50 mm,trim={0 0 0 20mm},clip]{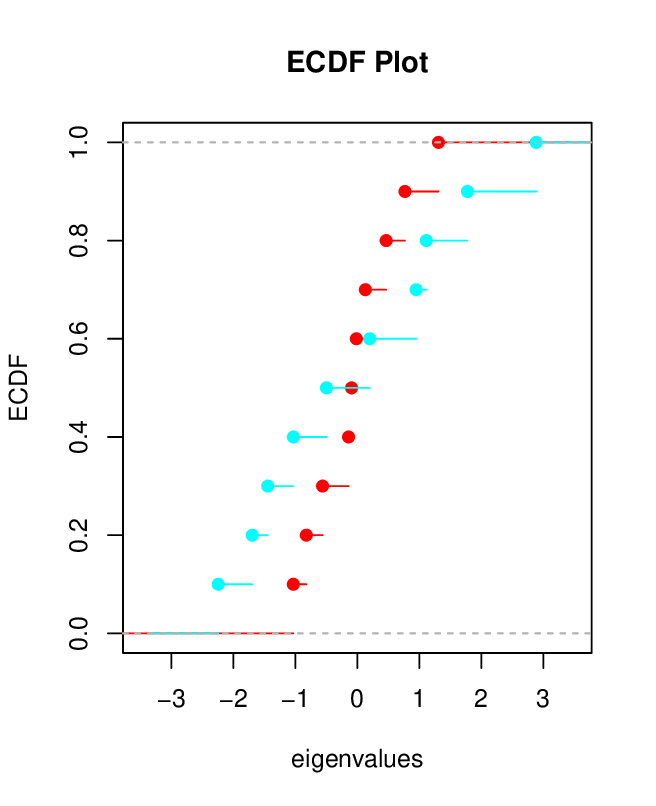} \hspace{8mm}
\includegraphics[height=60mm, width = 50 mm,trim={0 0 0 20mm},clip]{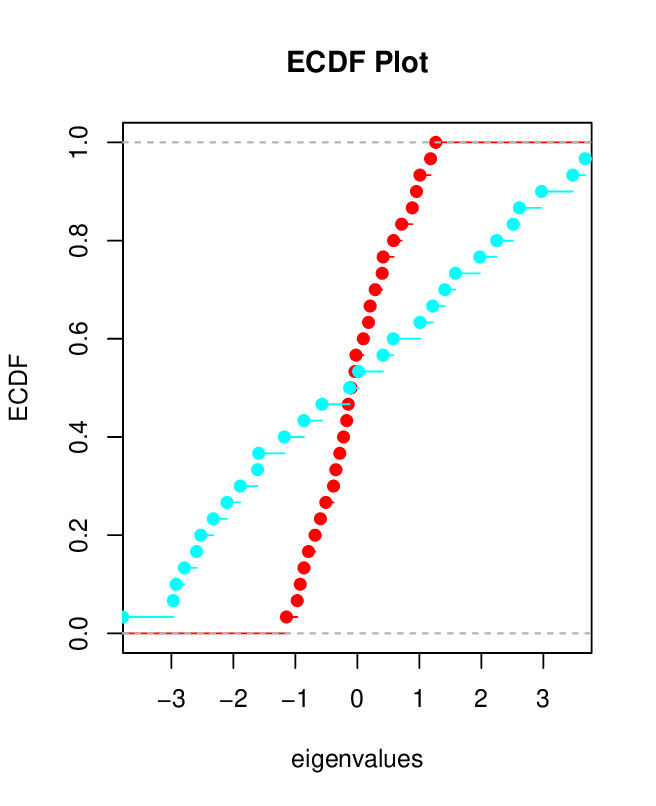} \hspace{8mm}
\includegraphics[height=60mm, width = 50 mm,trim={0 0 0 20mm},clip]{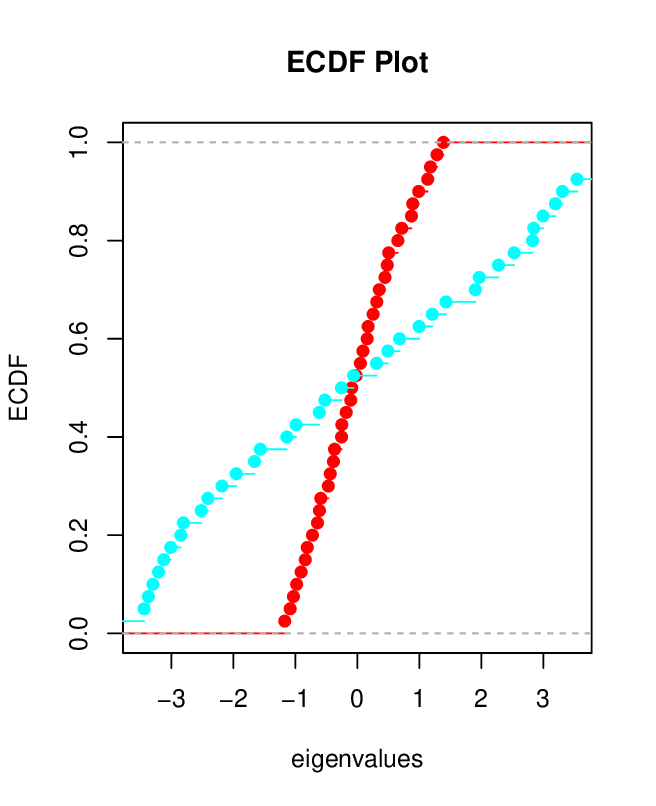} \vspace{0cm}
\caption{ECDFs of $\sqrt{n/p}(\mathbb{T}_Z-D(\mathbb{T}_Z))$ (cyan) and $\sqrt{n/p}(\mathbb{T}_X-D(\mathbb{T}_X))$ (red) for Example \ref{example: 1} setting; Columns are for $(n,p)=(100,10),(900,30),(1600,40)$ respectively  and rows correspond to $\alpha=0,1,2$ respectively. \label{fig: 5}}
\end{figure}

\section{Technical details} \label{sec: pf}

\subsection{Proof of Theorem \ref{thm: A}}
\noindent Since the entries of $\bT$ are $U$-statistics, their asymptotic behavior is governed by Hoeffding projections (see \citet{bose2018u} for details). As we will show, the first-order projection suffices for our purposes. We decompose $\bT$ as $\bT = \bT_1 + \bT_2$, where the entries of $\bT_1$ correspond essentially to the first-order projections, while $\bT_2$ captures the remainder. It will be shown that the term $\bT_2$ is negligible in the sense that it does not affect the LSD. Consequently, after appropriate centering and scaling, the LSDs of $\bT$ and $\bT_1$ are identical. 
\subsubsection{Hoeffding's decomposition of  $\bT$} 
Let $\bsX_i = (X_{1i},X_{2i},\ldots,X_{pi})^\prime$ for all $1 \leq i \leq n$. 
By Assumption 1, for $1 \leq k \neq l \leq p$ and $1 \leq i \neq j \leq n$, we have
\begin{eqnarray}
\bE(h((X_{ki},X_{li}),(X_{kj},X_{lj}))|\bsX_i)
&=& \bE(\Sg(X_{ki}-X_{kj})|\bsX_i)\bE(\Sg(X_{li}-X_{lj})|\bsX_i) = Y_{k,i,j}Y_{l,i,j},\hspace{0 cm} \label{eqn: a0}\\
\bE(h((X_{ki},X_{li}),(X_{kj},X_{lj}))|\bsX_j) &=& \bE(\Sg(X_{ki}-X_{kj})|\bsX_j)\bE(\Sg(X_{li}-X_{lj})|\bsX_j) = Y_{k,j,i}Y_{l,j,i}. \nonumber 
\end{eqnarray}
\begin{eqnarray}
\text{Note that} \quad \sum_{i,j\in [n], i \neq j} Y_{k,i,j}Y_{l,i,j}  = \sum_{i, j \in [n], i \neq j}  Y_{k,j,i}Y_{l,j,i} = \sum_{i, j  \in [n]}  Y_{k,i,j}Y_{l,i,j} - \sum_{i \in [n]}Y_{k,i,i}Y_{l,i,i}. \nonumber 
\end{eqnarray}
For $k, l \in [p]$, define
\begin{eqnarray}
G_{kl} = \dfrac{1}{n(n-1)}\sum_{i, j  \in [n]}  Y_{k,i,j}Y_{l,i,j}, \quad \text{and} \quad F_{kl} = \begin{cases} \dfrac{1}{n(n-1)}
\displaystyle{\sum_{i \in [n]}Y_{k,i,i}Y_{l,i,i}},\  \ \text{if $k \neq l$}, \\
0\  \ \text{otherwise}.
\end{cases} \label{eqn: defgf}
\end{eqnarray}
Let $\bG$ and $\bF$ be the matrices whose $(k,l)$-th entry are $G_{kl}$ and $F_{kl}$ respectively. 
Let $\bG^0$ be a diagonal matrix whose $k$-th diagonal entry is 
$$\frac{1}{n(n-1)}\sum_{i,j \in [n]}\bE(Y_{k,i,j}^2) = \bE(G_{kk}).$$ 
\noindent For $1 \leq k \neq l \leq p$ and $1 \leq i \neq j \leq n$,  define  
\begin{eqnarray}
\tilde{H}_{kl}(\bsX_i,\bsX_j) &=& h((X_{ki},X_{li}),(X_{kj},X_{lj})) - Y_{k,i,j}Y_{l,i,j} -  Y_{k,j,i}Y_{l,j,i} \quad \text{and} \quad \nonumber \\
H_{kl} &=& \begin{cases} \frac{1}{n(n-1)}\displaystyle{\sum_{1 \leq i \neq j \leq n}}\tilde{H}_{kl}(\bsX_i,\bsX_j),\ \ \text{if $k \neq l$}, \\
0,\ \ \text{otherwise}. \nonumber 
\end{cases}
\end{eqnarray}
 Let $\bH$ be the matrix whose $(k,l)$-th entry is $H_{kl}$ and $\bI_p$ be the identity matrix of order $p$. 
 
\noindent Therefore, we can write  $\sqrt{np^{-1}}(\bT - \rm{D}(\bT)) = \bT_1 +\bT_2$
\  \text{where}\ 
\begin{eqnarray}\label{eq: hoeffding}
&& \bT_1 = 2\sqrt{np^{-1}}(\bG  - g_1\bI_p),\ \ \ \text{and}\ \ \ \bT_2 = \sqrt{np^{-1}}(- 2\bF + \bH - 2(\rm{D}(\bG)-\bG^0)-2( \bG^0-g_1\bI_p)).  
\end{eqnarray}

\subsubsection{Negligibility of $\bT_2$} \label{sec: hoeffneg}
\noindent To show that the  L\'{e}vy distance between 
$\nu^{\sqrt{np^{-1}}(\bT-\rm{D}(\bT))}$ and  $\nu^{\bT_1}$ converges to $0$ almost surely, we use Lemma \ref{lem: bai} quoted from Corollary A.41 of \cite{bai2010spectral}.
Let $L(\nu^{\mathbb{A}},\nu^{\mathbb{B}})$ denote the L\'{e}vy distance between $\nu^{\mathbb{A}}$ and $\nu^{\mathbb{B}}$.
For two sequences $\{a_n:\ n \geq 1\}$ and $\{b_n:\ n \geq 1\}$, $a_n \lesssim b_n$ shall mean that, for some $C>0$, which does not depend on $n$, $a_n \leq Cb_n$ for all $n \geq 1$. 
\begin{lemma} \label{lem: bai}
Let $\mathbb{A}$ and $\mathbb{B}$ be $p\times p$ symmetric matrices. Then $(L(\nu^{\mathbb{A}},\nu^{\mathbb{B}}))^3 \leq p^{-1} \tr(\mathbb{A}-\mathbb{B})^2.$ 
\end{lemma}
\noindent To show that $\bT_2$ is negligible,
by Lemma \ref{lem: bai}, it is enough to show that 
\begin{eqnarray}
np^{-2}\tr(2\bF - \bH + 2\rm{D}(\bG)- 2g_1\bI_p)^2 \to 0\ \ \text{almost surely  as $p/n \to 0$}. \label{eqn: a1}
\end{eqnarray}
Note that $np^{-2}\tr(2\bF+2\rm{D}(\bG)-2g_1\bI_p-\bH)^2$
\begin{eqnarray*}
&\lesssim& np^{-2}\tr(\bF^2) + np^{-2}\tr((D(\bG)-\bG^0)^2) + np^{-2}\tr(\bH^2) + np^{-2}\tr(\bG^0-g_1\bI_p)^2 = A+B+C+D, \ \text{{\rm say}}.
\end{eqnarray*} We now show that each of these terms goes to $0$ almost surely. 
\vskip5pt

\noindent \textbf{Proof of $A\to 0$ almost surely}. Define
\begin{eqnarray}
K_1 = \{(k_1,k_2,k_3,k_4):\ \ 1\leq k_1 \neq k_2,  k_3 \neq k_4\leq p\}, \ \ 
I_1 = \{\{(i_1,j_1), (i_2,j_2), (i_3,j_3), (i_4, j_4)\}:\ \ 1 \leq i_r \neq j_r \leq n,\ 1 \leq r \leq 4\}. \nonumber 
\end{eqnarray}
Therefore,
\begin{eqnarray}
\frac{\bE(\tr(\bH^2))^2}{(n(n-1))^{-4}} = \sum_{(k_1, k_2, k_3, k_4)\in K_1}\bE(H_{k_1k_2}^2 H_{k_3,k_4}^2) \label{eqn: a2}  = \sum_{K_1, I_1} \bE(\tilde{H}_{k_1k_2}(\bsX_{i_1},\bsX_{j_1})\tilde{H}_{k_1k_2}(\bsX_{i_2},\bsX_{j_2})\tilde{H}_{k_3k_4}(\bsX_{i_3},\bsX_{j_3})\tilde{H}_{k_3k_4}(\bsX_{i_4},\bsX_{j_4})). \nonumber 
\end{eqnarray}
Note that, 
for all $1 \leq k \neq l \leq p$ and $1 \leq i \neq j \leq n$, by Assumptions 1 and 2, we have 
\begin{eqnarray}
 \bE(\Sg(X_{ki}-X_{kj})) &=& \bP(X_{ki}>X_{kj}) - \bP(X_{ki}<X_{kj}) = 0, \ \ \text{by Assumption 2},\ \ \ \label{eqn: d1} \\
\bE(h((X_{ki},X_{li}),(X_{kj},X_{lj}))) &=& \bE(\Sg(X_{ki}-X_{kj}))\bE(\Sg(X_{li}-X_{lj})) = 0, \label{eqn: d2} \ \  
\text{by Assumption 1 and (\ref{eqn: d1})}, \nonumber  \\
 \bE(Y_{k,i,j}) &=& \bE(\Sg(X_{ki}-X_{kj})) = 0, \ \ \text{by (\ref{eqn: d1})},  \label{eqn: d3} \\
 \bE(Y_{k,i,j} Y_{l,i,j}) &=& \bE(Y_{k,i,j})\bE(Y_{l,i,j}) = 0, \ \ \text{by Assumption 1 and (\ref{eqn: d3})}.\label{eqn: d4} 
\end{eqnarray}
Therefore, by (\ref{eqn: d2}) and (\ref{eqn: d4}), we have
\begin{eqnarray}
\bE(\tilde{H}_{kl}(\bsX_{i},\bsX_{j})) = \bE(h((X_{ki},X_{li}),(X_{kj},X_{lj}))- Y_{k,i,j}Y_{l,i,j} -  Y_{k,j,i}Y_{l,j,i}) =  0.  \label{eqn: d5}
\end{eqnarray}
Using Equation (\ref{eqn: d5}) we  can conclude that 
if there is at least one $s \in \{1,2,3,4\}$ such that neither the value of $i_s$ nor the value of $j_s$ repeats in $\cup_{w=1, w \neq s}^{4} \{i_w,j_w\}$, then 
\begin{eqnarray}
&& \bE(\tilde{H}_{k_1k_2}(\bsX_{i_1},\bsX_{j_1})\tilde{H}_{k_1k_2}(\bsX_{i_2},\bsX_{j_2})\tilde{H}_{k_3k_4}(\bsX_{i_3},\bsX_{j_3})\tilde{H}_{k_3k_4}(\bsX_{i_4},\bsX_{j_4}))  = 0. \label{eqn: d6}
\end{eqnarray}
Let $\tilde{Z}$ be a random variable which is independent of $X_{kj}$ and $X_{lj}$. Then 
\begin{eqnarray}
\bE(\tilde{H}_{kl}(\bsX_{i},\bsX_{j})\tilde{Z}) &=& \bE(\bE(\tilde{H}_{kl}(\bsX_{i},\bsX_{j})|\bsX_i)\bE(\tilde{Z}|\bsX_i)) = \bE(\bE((h((X_{ki},X_{li}),(X_{kj},X_{lj}))-Y_{k,i,j}Y_{l,i,j} - Y_{k,j,i}Y_{l,j,i})|X_{i}) \bE(\tilde{Z}|\bsX_i)) \nonumber \\
&&\hspace{-1 cm} = \bE((\bE(h((X_{ki},X_{li}),(X_{kj},X_{lj}))|\bsX_i)-Y_{k,i,j}Y_{l,i,j} - \bE(Y_{k,j,i}Y_{l,j,i})) \bE(\tilde{Z}|\bsX_i)) = 0,\ \ \text{by (\ref{eqn: a0}) and (\ref{eqn: d4})}  \ \   \label{eqn: d7}
\end{eqnarray}
 for $1 \leq i \neq j \leq n$ and $1 \leq k \neq l \leq p$. The third equality holds since $Y_{k,i,j}$ is a function of $X_{ki}$ and hence     $\bE(Y_{k,i,j}Y_{l,i,j}|\bsX_i) = Y_{k,i,j}Y_{l,i,j}$ and $\bE(Y_{k,j,i}Y_{l,j,i}|\bsX_i) = \bE(Y_{k,j,i}Y_{l,j,i})$ almost surely.

From Equation (\ref{eqn: d7}) we conclude that if there is at least one $s \in \{1,2,3,4\}$ such that either the value of $i_s$ or the value of $j_s$ does not repeats in $\cup_{w=1, w \neq s}^{4}\{i_w,j_w\}$, then 
\begin{eqnarray}
&&\bE(\tilde{H}_{k_1k_2}(\bsX_{i_1},\bsX_{j_1})\tilde{H}_{k_1k_2}(\bsX_{i_2},\bsX_{j_2})\tilde{H}_{k_3k_4}(\bsX_{i_3},\bsX_{j_3})\tilde{H}_{k_3k_4}(\bsX_{i_4},\bsX_{j_4}))  = 0. \label{eqn: d8}
\end{eqnarray}
Let 
\(I_2 = \{\{(i_1,j_1), (i_2,j_2), (i_3,j_3), (i_4, j_4)\} \in I_1:\ \ |\{i_1,j_1, i_2,j_2, i_3,j_3, i_4, j_4\}|\leq 4\}. \) 
By (\ref{eqn: a2}), we have
\begin{eqnarray}
n^2p^{-2}\bE(\tr(\bH^2))^2 &=& n^2p^{-2}(n(n-1))^{-4}(T_1 + T_2) \quad \text{where} \label{eqn: d9} \\
&& \hspace{-3 cm}T_1 =\sum_{K_1,\  I_2} \bE(\tilde{H}_{k_1k_2}(\bsX_{i_1},\bsX_{j_1})\tilde{H}_{k_1k_2}(\bsX_{i_2},\bsX_{j_2})\tilde{H}_{k_3k_4}(\bsX_{i_3},\bsX_{j_3})\tilde{H}_{k_3k_4}(\bsX_{i_4},\bsX_{j_4})), \nonumber \\
&& \hspace{-3 cm} T_2 = \sum_{K_1, \ I_1\setminus I_2} \bE(\tilde{H}_{k_1k_2}(\bsX_{i_1},\bsX_{j_1})\tilde{H}_{k_1k_2}(\bsX_{i_2},\bsX_{j_2})\tilde{H}_{k_3k_4}(\bsX_{i_3},\bsX_{j_3})\tilde{H}_{k_3k_4}(\bsX_{i_4},\bsX_{j_4})). \nonumber 
\end{eqnarray}
Equations (\ref{eqn: d6}) and (\ref{eqn: d8}) clearly implies $T_2 = 0$ and therefore 
\begin{eqnarray}
n^2p^{-2}\bE(\tr(\bH^2))^2 = n^2p^{-2}(n(n-1))^{-4}T_1. \label{eqn: d10}
\end{eqnarray} 
Moreover, as $\tilde{H}_{kl}(\bsX_i,\bsX_j)$ is bounded for all $k,l \in [p]$ and $i,j \in [n]$, we clearly have 
\begin{eqnarray}
T_1 = O(n^4p^4). \label{eqn: d11}
\end{eqnarray}
Equations (\ref{eqn: d10}) and (\ref{eqn: d11}) together imply
\(\bE(np^{-1}\tr(\bH^2))^2  = O((p/n)^2). \) 
As $p/n \to 0$, we have 
\(\bE((np^{-2}\tr(\bH^2) )^2 \leq Cp^{-2}\ \ \text{for some $C>0$}.\) 
\noindent This completes the proof of $A\to 0$ almost surely.

\noindent \textbf{Proof of $B \to 0$ almost surely}.  Define $I_3 = \{(i_1,i_2,i_3,i_4):\ i_s\in [n],\ 1 \leq s \leq 4\}$. Therefore,
\begin{eqnarray}
(n(n-1))^4 \bE(\tr(\bF^2))^2 = \sum_{K_1,\ I_3}\bE( Y_{k_1,i_1,i_1}Y_{k_2,i_1,i_1}Y_{k_1,i_2,i_2}Y_{k_2,i_2,i_2}Y_{k_3,i_3,i_3}Y_{k_4,i_3,i_3}Y_{k_3,i_4,i_4}Y_{k_4,i_4,i_4}). \label{eqn: d12}
\end{eqnarray}
As $Y_{k,i,j}$ is bounded for all  $k \in [p]$, and 
$i, j \in [n]$, 
by (\ref{eqn: d12}), we have $\bE(np^{-1}\tr(\bF^2))^2 = O((p/n)^2)$ and hence 
\begin{eqnarray}
\bE((np^{-2}\tr(\bF^2) )^2 \leq Cp^{-2}\ \ \text{for some $C>0$}. 
\end{eqnarray}
This completes the proof of $B \to 0$ almost surely.
\vskip 5pt

\noindent \textbf{Proof of $C \to 0$ almost surely}.  
Let $\tilde{Y}_{k,i,j} = Y_{k,i,j}^2  - \bE(Y_{k,i,j}^2)$. Then by (\ref{eqn: defgf}),  $(n(n-1))^{4}\bE(\tr((\rm{D}(\bG)-\bG^0)^2))^2$ equals
\begin{eqnarray}
&&\hspace{-4 cm} =\sum_{k_1,k_2=1}^{p} \sum_{\stackrel{(i_1,i_2,i_3,i_4) \in I_3}{(j_1,j_2,j_3,j_4) \in I_3}} \bE(\tilde{Y}_{k_1,i_1,j_1} \tilde{Y}_{k_1,i_2,j_2} \tilde{Y}_{k_2,i_3,j_3} \tilde{Y}_{k_2,i_4,j_4}). \label{eqn: d14}
\end{eqnarray}
As $\bE(\tilde{Y}_{k,i,j}) = 0$ for all $k \in [p]$ and $i, j\in [n]$, 
if there is at least one $s \in \{1,2,3,4\}$ such that
the value $i_s$ does not repeat in $\{i_w:\ 1 \leq w \leq 4,\ w \neq s\}$, then  
clearly we have
\begin{eqnarray}
&& \bE(\tilde{Y}_{k_1,i_1,j_1} \tilde{Y}_{k_1,i_2,j_2} \tilde{Y}_{k_2,i_3,j_3} \tilde{Y}_{k_2,i_4,j_4}) = 0. \label{eqn: d15} 
\end{eqnarray} 
Let 
$I_4  = \{(i_1,i_2,i_3,i_4) \in I_3:\ |\{(i_1,i_2,i_3,i_4\}| \leq 2\}.$ 
Then by (\ref{eqn: d14}), we have
\begin{equation}\label{eqn: d16}
(n(n-1))^{4}\bE(\tr((\rm{D}(\bG)-\bG^0)^2))^2 = T_3 + T_4,
\end{equation}
where
\begin{eqnarray}
 T_3 = \sum_{k_1,k_2=1}^{p} \ \sum_{\stackrel{(j_1,j_2,j_3,j_4) \in I_3}{(i_1,i_2,i_3,i_4) \in I_4}} \bE(\tilde{Y}_{k_1,i_1,j_1} \tilde{Y}_{k_1,i_2,j_2} \tilde{Y}_{k_2,i_3,j_3} \tilde{Y}_{k_2,i_4,j_4}), \ \  T_4 = \sum_{k_1,k_2=1}^{p} \sum_{\stackrel{(j_1,j_2,j_3,j_4) \in I_3}{(i_1,i_2,i_3,i_4) \in I_3\setminus I_4}} \bE(\tilde{Y}_{k_1,i_1,j_1} \tilde{Y}_{k_1,i_2,j_2} \tilde{Y}_{k_2,i_3,j_3} \tilde{Y}_{k_2,i_4,j_4}). \nonumber 
\end{eqnarray}
Equation (\ref{eqn: d15}) implies $T_4 =0$, and hence 
\begin{eqnarray} 
 (n(n-1))^{4}\bE(\tr((\rm{D}(\bG)-\bG^0)^2))^2 = T_3. \label{eqn: d17}
\end{eqnarray}
Moreover, as $Y_{k,i,j}$ is bounded for all $k \in [p]$ and $i, j \in [n]$, 
we have 
\begin{eqnarray}
T_3 = O(p^2n^6). \label{eqn: d18}
\end{eqnarray}
Eqns. (\ref{eqn: d17}) and (\ref{eqn: d18}) imply $\bE(np^{-1}\tr(({\rm D}(\bG)-\bG^0)^2))^2  = O(1)$ and as $p/n \to 0 $, 
\(\bE(np^{-2}\tr(({\rm D}(\bG)-\bG^0)^2))^2  \leq Cp^{-2}\ \ \text{for some $C>0$}. \) 
This completes the proof of $C\to 0$ almost surely.
\vskip 5pt

\noindent \textbf{Proof of $D\to 0$ almost surely}. Note that $np^{-2}\tr(\bG^0-g_1\bI_p)^2 = np^{-2}\tr((\bG^0)^2)+ np^{-1} g_1^2- 2g_1 np^{-2}\tr(\bG^0)$. Also, by Assumption 3, we have 
\begin{eqnarray}
  \frac{n}{p^{2}}\tr(\bG^0) &=&   \frac{n}{p^{2}} \sum_{k=1}^p \bE(\bG_{kk}) = \frac{(n(n-1))^{-1}n}{p^{2}} \sum_{k=1}^p\sum_{i,j\in [n]} \mathbb{E}(Y_{k,i,j}^2) = \frac{(n(n-1))^{-1}n}{p^{2}} \sum_{k=1}^p\sum_{i\in [n]} \tr(\bG_{k,i}) = (n/p)g_1 + o(1). \nonumber  
\end{eqnarray}
\noindent Next,
\begin{eqnarray}
    \frac{n}{p^2}\tr(\bG^0)^2 &=& \frac{n}{p^2} \sum_{k=1}^p (\bE(\bG_{kk}))^2 = \frac{(n(n-1))^{-2}n}{p^2} \sum_{k=1}^p\sum_{i_1,j_1,i_2,j_2 \in [n]} \mathbb{E}(Y_{k,i_1,j_1}^2)\mathbb{E}(Y_{k,i_2,j_2}^2) \nonumber \\
    && \hspace{-2 cm} = \frac{(n(n-1))^{-2}n}{p^2} \sum_{k=1}^p\sum_{i_1,i_2\in [n]} \tr(\bG_{k,i_1})\tr(\bG_{k,i_2}) = \frac{n^3}{(n(n-1))^2}\frac{n^2}{p^2}\sum_{k=1}^p \Big(n^{-2}\sum_{i=1}^n\tr(\bG_{k,i})\Big)^2 = T_7 + T_8 +T_9
 \nonumber \\
 && \hspace{-2 cm} = \frac{n^3}{(n(n-1))^2}\frac{n^2}{p^2}\sum_{k=1}^p \Big(n^{-2}\sum_{i=1}^n\tr(\bG_{k,i})-g_1\Big)\Big(n^{-2}\sum_{i=1}^n\tr(\bG_{k,i})\Big) + g_1\frac{n^3}{(n(n-1))^2}\frac{n^2}{p^2}\sum_{k=1}^p \Big(n^{-2}\sum_{i=1}^n\tr(\bG_{k,i})\Big)\nonumber \\
 && \hspace{-2 cm} = O\left((n/p^2) \sum_{k=1}^p \Big(n^{-2}\sum_{i=1}^n\tr(\bG_{k,i})-g_1\Big) \right) +  g_1 O\left( (np^2)^{-1}\sum_{k=1}^p\sum_{i\in [n]} \tr(\bG_{k,i}) \right) \nonumber \\
 && \hspace{-2 cm} = O\bigg[(n/p) \left(p^{-1}n^{-2} \sum_{k=1}^p\sum_{i\in [n]} \tr(\bG_{k,i})  - g_1\right)\bigg] + (n/p)g_1^2 + o(1) = (n/p)g_1^2 + o(1). \nonumber 
\end{eqnarray}
Thus  $np^{-2}\tr(\bG^0-g_1\bI_p)^2 \to 0$. This completes the proof of $D\to 0$ almost surely.

\subsubsection{LSD of $\bT_1$}
\noindent The following Lemma is used to establish the LSD of appropriately centered and scaled $\bT_1$ and for its identification.
\begin{lemma}\label{lem:bose} (Lemma 1.2.4 in \citet{bose2018patterned})
Let $\{\mathbb{A}_p:\ p \geq 1\}$ be a sequence of real symmetric random matrices. \\ Suppose that 
\vskip 3pt

\noindent \textbf{(c1)} $p^{-1}\bE\tr(\mathbb{A}_p^R) \to \lambda_R$ for all $R \geq 1$;
\vskip 3pt

\noindent \textbf{(c2)} $\bE(p^{-1}\tr(\mathbb{A}_p^R) -p^{-1}\bE\tr(\mathbb{A}_p^R) )^4 = O(p^{-2})$ for all $R \geq 1$;
\vskip 3pt

\noindent \textbf{(c3)} the moment sequence $\{\lambda_R: R \geq 1\}$ satisfies $\sum_{R=1}^{\infty} \lambda_{2R}^{-1/2R} = \infty$ (Carleman's condition), and hence determines a unique probability distribution. 
\vskip 3pt
\noindent Then $\nu^{\bA_p}$ converges almost surely to the distribution determined by the moment sequence $\{\lambda_R\}$.
\end{lemma}

\noindent Lemma \ref{lem:bose} is now applied to sequence of matrices $\bT_1$ in order to determine the LSD. In this section, we respectively verify that $\bT_1$ satisfies conditions (c1) and (c2). Condition (c3) is readily verified for the moments of a semi-circle  variable. Theorem \ref{thm: 1} then follows.

\paragraph{Verification of (c1) condition for $\bT_1$} \label{sec: pg1}
\noindent The following lemma immediately implies that $\bT_1$ satisfies condition (c1).
\begin{lemma} \label{lem: lem2}
\noindent  Let Assumptions 1,2, G1 and G2 hold and $n,p \to \infty, p/n \to 0$. 
Then, for all $R \geq 1$, \\ $\lim p^{-1}\bE\tr(\big(\sqrt{n/p}\ (\bG- g_1\bI_p)\big)^{2R-1}) = 0$ and
$\lim p^{-1}\bE\tr(\big(\sqrt{n/p}\ (\bG- g_1\bI_p)\big)^{2R})=\sum_{\pi \in {\rm{NC}}_2(2R)} g_{2\pi}$. 
\end{lemma}

\noindent \textbf{Proof of Lemma \ref{lem: lem2}}: \hspace{1mm}
For any positive integer $R$,  let $\sigma_R$ denote the cyclic permutation on $\{1,2,\ldots R \}$ such that $\sigma_R(r)=r+1$, $1\leq r\leq R-1$, and $\sigma_{R}(R)=1$. Now   $ {p^{-1}}\bE\big[\tr\big(\sqrt{np^{-1}}(\bG-g_1\bI_p)\big)^R\big]$ equals
\begin{eqnarray}
   &&\frac{n^{R/2}}{p^{(R/2)+1}}\sum_{\substack{k_r\in[p]\\r \in[R]}}\bE\big[\prod_{r \in [R]}\big(G_{k_r,k_{\sigma_R(r)}}-g_1\delta_{k_rk_{\sigma_R(r)}}\big)\big] =\frac{n^{R/2}}{p^{(R/2)+1}(n(n-1))^R}\sum_{\substack{k_r\in[p]\\r \in[R]}}\bE\bigg[\prod_{r \in [R]}\sum_{\substack{i_r,j_r\in[n]}}\big(Y_{k_r,i_r,j_r}Y_{k_{\sigma_R(r),i_r,j_r}}-g_1\delta_{k_r,k_{\sigma_R(r)}}\big)\bigg] \nonumber\\
   && =\frac{n^{R/2}}{p^{(R/2)+1}(n(n-1))^R}\sum_{\substack{k_r\in[p]\\r \in[R]}}\sum_{\substack{i_r,j_r\in[n]\\r \in[R]}}\bE\bigg[\prod_{r \in [R]}\big(Y_{k_r,i_r,j_r}Y_{k_{\sigma_R(r),i_r,j_r}}-g_1\delta_{k_r,k_{\sigma_R(r)}}\big)\bigg].    \label{eq:mainnew}
\end{eqnarray}
Consider the connected bipartite graph between the  $k$-indices $K=\{k_r: r \in [R]\}$ and $i$-indices $I=\{i_r:r \in[R]\}$ with edges $E =\{(k_{r+\epsilon_r},i_r): r \in [R],\epsilon=0,1\}$. A single edge in $E$  cannot occur in isolation. Suppose there exists an edge $(k_{q+\epsilon},i_q) \in E$ which does not match with any other edges of $E$ for some $q$ and $\epsilon=0,1$. Then, clearly $k_{q+\epsilon} \neq k_{q+1-\epsilon}$ and (\ref{eq:mainnew}) is reduced to the following
\begin{eqnarray}
&& \frac{n^{R/2}}{p^{(R/2)+1}(n(n-1))^R}\sum_{\substack{k_r\in[p]\\r \in[R]}}\sum_{\substack{i_r,j_r\in[n]\\r \in[R]}} \bE(Y_{k_{q+\epsilon},i_q,j_q})\bE\bigg[Y_{k_{q+1-\epsilon},i_q,j_q}\prod_{\substack{r \in [R] \\ r \neq q}}\big(Y_{k_r,i_r,j_r}Y_{k_{\sigma_R(r),i_r,j_r}}-g_1\delta_{k_r,k_{\sigma_R(r)}}\big)\bigg] \nonumber 
\end{eqnarray}
which is $0$, since $\bE(Y_{k_r,i_r,j_r})=0$ by Assumptions 1 and 2.  This means that the value of (\ref{eq:mainnew}) is $0$ whenever an edge in $E$ exists in isolation.
Hence it is sufficient to consider only those edges which repeat at least twice. There are at most $R$ such distinct edges and since the graph is connected, we have
\begin{eqnarray}
    |K|+|I| \leq |E|+1=R+1, \label{eq: numb}
\end{eqnarray}
\begin{eqnarray}
    p^{-1}\bE\big[\tr\big(\sqrt{np^{-1}}(\bG-g_1\bI_p)\big)^R\big]=\frac{p^{-((R/2)+1)}n^{R/2}}{(n(n-1))^R}\sum_{\substack{{K,I}\\|E|\leq R}}\sum_{\substack{j_r\in[n] \\ r \in [R]}}\bigg(\prod_{r \in[R]}\big(Y_{k_r,i_r,j_r}Y_{k_{\sigma_R(r),i_r,j_r}}-g_1\delta_{k_r,k_{\sigma_R(r)}}\big)\bigg). \label{eq: contriterm}
\end{eqnarray}
Also note note that from (\ref{eq: numb}) we have
\begin{eqnarray}
    p^{-1}\bE\big[\tr\big(\sqrt{np^{-1}}(\bG-g_1\bI_p)\big)^R\big]\lesssim O\big((p/n)^{((R/2)-|I|)}\big). \label{eqn: kknew}
\end{eqnarray}
Since $p/n \to 0$ as $p,n \to \infty$, for any sort of contribution from (\ref{eq: contriterm}) the following must hold:
\(    |K| = (R/2)+1,\ \ \
    |I|=(R/2). \) 
This means that for (\ref{eq: contriterm}) to contribute, not only are the restrictions on $|K|,|I|$ necessary, it is also essential that the number of distinct edges $|E|$ is exactly equal to $R$. This means that each edge is repeated exactly twice.  Let $P_2(2R)$ be the set of all paired partitions of $\{1,2,\ldots 2R\}$.   
It is easily verifiable that the set of all connected bipartite graphs for which $|K|+|I| =R+1$ has a bijection with $P_2(2R)$.  Let $K_{\pi}$ and $I_{\pi}$ be respectively the set of distinct $k$-indices and $i$-indices induced by $\pi \in P_2(2R)$. 
\begin{eqnarray}
    p^{-1}\bE\big[\tr\big(\sqrt{np^{-1}}(\bG-g_1\bI_p)\big)^R\big]&=&\frac{p^{-((R/2)+1)}n^{R/2}}{(n(n-1))^R}\sum_{\substack{\pi\in P_2(2R)}}\sum_{\substack{K_{\pi},I_{\pi}\\j_r\in[n],r\in[R]}}\bE\bigg(\prod_{r \in[R]}\big(Y_{k_r,i_r,j_r}Y_{k_{\sigma_R(r),i_r,j_r}}-g_1\delta_{k_r,k_{\sigma_R(r)}}\big)\bigg). \hspace{0.5 cm} \label{eq: contri1}
\end{eqnarray}
Consider the case when there exists at least one edge matching of the type $(k_l,i_l)=(k_{\sigma_R(l)},i_l)$ for some $l \in [R]$.  Without loss of generality, assume $l=1$. Then, the contribution of this matching in $p^{-1}\bE\big[\tr\big(\sqrt{np^{-1}}(\bG-g_1\bI_p)\big)^R\big]$ equals
\begin{eqnarray}
&&  \bigg|\sqrt{(n/p)}n^{-2} \sum_{k_1,i_1,j_1} \bE(Y_{k_1,i_1,j_1}Y_{k_1,i_1,j_1} - g_1) p^{-1}\bE\Big(\big(\sqrt{np^{-1}}(\bG-g_1\bI_p)\big)^{R-1}_{k_1,k_1}\Big)\bigg|  \nonumber \\
& = & \bigg| \sqrt{\frac{p}{n}} \sum_{k_1=1}^p \frac{n}{p} (n^{-2}\sum_{i_1=1}^n \tr(\bG_{k_1,i_1}) - g_1) \frac{1}{p}\bE\Big(\big(\sqrt{\frac{n}{p}}(\bG-g_1\bI_p)\big)^{R-1}_{k_1,k_1}\Big)\bigg| \leq  \sqrt{np} \big| \frac{1}{n^2 p} \sum_{k_1=1}^p  \sum_{i_1=1}^n \tr(\bG_{k_1,i_1}) - g_1 \big| = o(1), \nonumber 
\end{eqnarray}
where $p^{-1}\bE\Big(\big(\sqrt{np^{-1}}(\bG-g_1\bI_p)\big)^{R-1}_{k_1,k_1}\Big) = O(1)$  proceeds by the same reasoning as in equation (\ref{eqn: kknew}).
This implies that asymptotically there is no contribution from  matching of the type $(k_l,i_l)=(k_{\sigma_R(l)},i_l)$ for some $l \in [R]$.  Therefore only other non-crossing paired partitions and crossing partitions of $[2R]$ are considered for analysis. Let $\bI(.)$ denote the indicator function and let $\zeta(r)=\bI(r\hspace{0.5mm} \text{is even})$. Therefore, (\ref{eq: contri1}) can be rewritten as 
 \begin{eqnarray}
      &&\frac{p^{-((R/2)+1)}n^{R/2}}{(n(n-1))^R}\sum_{\substack{\pi \in P_2(2R) }}\sum_{\substack{K_{\pi},I_{\pi}\\j_r\in[n],r\in[R]}}\ \bigg(\prod_{(b,s) \in \pi}\big(G_{k_{\lceil b/2 \rceil+\zeta(b)},i_{\lceil b/2 \rceil}}(j_{\lceil b/2 \rceil},j_{\lceil s/2 \rceil})\big)\delta_{k_{\lceil b/2 \rceil+\zeta(b)},k_{\lceil s/2 \rceil+\zeta(s)}}\delta_{i_{\lceil b/2 \rceil}i_{\lceil s/2 \rceil}}\bigg)\nonumber
 \end{eqnarray}
 Let $NC_2(2R)$ be the set of all non-crossing paired partitions of $[2R]$. 
 Similar to the arguments of Theorem 2.1.3 from \citet{bose2018patterned}, it can be shown that the contribution of the above term for $P_2(2R)/NC_2(2R)$ is $0$ in the limit. 
 This means that all the contribution stems from non crossing paired partitions of $2R$ i.e. $NC_2(2R)$.  Therefore, $p^{-1}\bE\big[\tr\big(\sqrt{np^{-1}}(\bG-\gamma_1\bI_p)\big)^R\big] $ equals
 \begin{eqnarray}
 \frac{p^{-((R/2)+1)}}{n^{3R/2}}\sum_{\substack{\pi \in NC_2(2R) }}\sum_{\substack{K_{\pi},I_{\pi}\\j_r\in[n],r\in[R]}}\ \bigg(\prod_{(b,s) \in \pi}\big(\bG_{k_{\lceil b/2 \rceil+\zeta(b)},i_{\lceil b/2 \rceil}}(j_{\lceil b/2 \rceil},j_{\lceil s/2 \rceil})\big)\delta_{k_{\lceil b/2 \rceil+\zeta(b)},k_{\lceil s/2 \rceil+\zeta(s)}}\delta_{i_{\lceil b/2 \rceil}i_{\lceil s/2 \rceil}}\bigg) + o(1). \label{eqn: kkknew}
 \end{eqnarray}
 \noindent Next, observe the following facts about elements of $I$.
 \\(a) Value of the index $i_r$ has to be repeated at least twice in $I$ for all $r \in [R] $. \\ 
 \noindent (b) $|I| = R/2$ (by (\ref{eqn: kknew})).\\
 \noindent (c) Only terms obtained from non-crossing paired partitions of $[2R]$ contribute.\\
\noindent  The fact (b) clearly indicates 
\(p^{-1}\bE\big[\tr\big(\sqrt{np^{-1}}(\bG-\gamma_1\bI_p)\big)^R\big]  = o(1)\ \ \text{if $R$ is odd}. \) 
Therefore only even values of R are considered from hereon. Also (a) - (c) together induce non-crossing paired partitions on $I$ with even $R$. Let $NC_2(I)$ be the set of all non-crossing paired partitions of $I$. Then $NC_2(2R)$ along with (a) - (c) has a bijection $F$ with $NC_2(I)$ where $k_b = k_{\sigma_{R}(s)}$ and $k_s = k_{\sigma_{R}(b)}$ with $(i_b,i_s)\in \pi \in NC_2(I)$. \vskip 3pt
\noindent  Then (\ref{eqn: kkknew}) reduces that  $\lim p^{-1}\bE\big[\tr\big(\sqrt{np^{-1}}(\bG-g_1\bI_p)\big)^R\big] $ equals
 \begin{eqnarray}
     \sum_{\pi \in NC_2(I)}{\rm{lim}}\frac{p^{-((R/2)+1)}}{n^{3R/2}}\sum_{K_{F^{-1}(\pi)}, I_{F^{-1}(\pi)}}\bigg(\prod_{(i_b,i_s)\in \pi}\big((Y_{k_b,i_b,j_b}Y_{k_{\sigma_R(b)},i_b,j_b}Y_{k_{\sigma_{R}(s)},i_b,j_s}Y_{k_s,i_b,j_s})(\delta_{k_bk_{\sigma_{R}(s)}}\delta_{k_sk_{\sigma_{R}(b)}}\delta_{i_bi_s})\big)\bigg). \label{eq: main}
 \end{eqnarray}
 Note that the particular matching $\delta_{k_bk_{\sigma_{R}(s)}}\delta_{k_sk_{\sigma_{R}(b)}}$ of elements of $K$ is considered as crossing partitions of elements of $E$ have no contribution.   Using definitions of $\bG_{k,i}$, the term in (\ref{eq: main}) reduces to the following \\
     \begin{eqnarray}
      &&\sum_{\pi \in NC_2(I)}{\rm{lim}}\frac{p^{-((R/2)+1)}}{n^{3R/2}}\sum_{K_{F^{-1}(\pi)}, I_{F^{-1}(\pi)}}\sum_{\substack{j_r \in[n]\\r \in [R]}}\bigg(\prod_{(i_b,i_s)\in \pi}\big(\bG_{k_b,i_b}(j_b,j_s)\bG_{k_{\sigma_R(b)},i_b}(j_s,j_b))(\delta_{k_bk_{\sigma_{R}(s)}}\delta_{k_sk_{\sigma_{R}(b)}}\delta_{i_bi_s})\big)\bigg)\nonumber\\
       &=&\sum_{\pi \in NC_2(I)}{\rm{lim}}\frac{p^{-((R/2)+1)}}{n^{3R/2}}\sum_{K_{F^{-1}(\pi)}, I_{F^{-1}(\pi)}}\bigg(\prod_{(i_b,i_s)\in \pi}\big(\tr(\bG_{k_b,i_b}\bG_{k_{\sigma_R(b)},i_b})(\delta_{k_bk_{\sigma_{R}(s)}}\delta_{k_sk_{\sigma_{R}(b)}}\delta_{i_bi_s})\big)\bigg)\nonumber\\
    &=&\sum_{\pi \in NC_2(I)}{\rm{lim}}\frac{p^{-((R/2)+1)}}{n^{3R/2}}\sum_{K_{F^{-1}(\pi)}}\bigg(\prod_{(i_b,i_s)\in \pi}\big(\sum_{i_b}\tr(\bG_{k_b,i_b}\bG_{k_{\sigma_R(b)},i_b})(\delta_{k_bk_{\sigma_{R}(s)}}\delta_{k_sk_{\sigma_{R}(b)}}\delta_{i_bi_s})\big)\bigg)\nonumber\\
    &=&\sum_{\pi \in NC_2(I)}{\rm{lim}}(n^{-(R/2)}p^{-1})\sum_{K_{F^{-1}(\pi)}}\bigg(\prod_{(i_b,i_s)\in \pi}\big(p^{-1}n^{-2}\sum_{i_b}\tr(\bG_{k_b,i_b}\bG_{k_{\sigma_R(b)},i_b})(\delta_{k_bk_{\sigma_{R}(s)}}\delta_{k_sk_{\sigma_{R}(b)}}\delta_{i_bi_s})\big)\bigg) = \sum_{\pi \in NC_2(I)} g_{2\pi}. \nonumber 
    \end{eqnarray}
    This proves Lemma \ref{lem: lem2}.  

\begin{remark} \label{rem: heinypf}
Suppose that $\{k_u : 1 \leq u \leq R\}$ induces a partition with $B$ blocks, and $\{i_u : 1 \leq u \leq R\}$ induces a partition with $R_1$ blocks. For each $u \in {1,\dots,R_1}$ and $v \in {1,\dots,B}$, let $\tilde{f}_{uv}$ denote the frequency of the $i$-index corresponding to the $u$-th block appears with the $k$-index associated with the $v$-th block. Thus, encountering the each $i$-index appears twice,   $\sum_{v=1}^B \tilde{f}_{uv}$ equals twice the cardinality of of the $u$-th block. Let $f_{tu} = \sum_{v=1}^B \mathbb{I}(f_{uv} = t)$.  Then, by the properties of $\tilde{Y}_{ki}$ mentioned in Remark \ref{rem: heiny} and noting that $\sum_{u=1}^{R_1}\sum_{t=1}^{2R} tf_{tu} =2R$,  for that particular partition we have, $(n/p)^{R/2} p^{-1}\tr(\tilde{G}^{\mathrm{IID}\ R}) = O(p^{-R/2-1+B}n^{R/2 + R_1- \sum_{u=1}^{R_1} \big[f_{1u} + \sum_{t=2}^{2R} f_{tu} (t/2)\big]}) = O(p^{-R/2-1+B}n^{\sum_{u=1}^{R_1} \big[1-(1/2)f_{1u} - \sum_{t=2}^{2R} f_{tu} (t/4)\big]})$ which is contributing only if $B = R/2+1$, $f_{2u}=2$, $f_{tu}=0$ for all $t \neq 2$ and  $u$ and consequently $R_1 = R/2$.  Clearly, this will direct us to (\ref{eqn: kkknew}) after replacing $\mathbb{G}_{k,i}$'s by $n \mathbb{G}_{k,i}$. In this case, $\mathbb{V}{\rm{ar}}(Y_{k1})$ will replaced by $1$.  Thus \eqref{eqn: g1A} holds with $g_1=1$, and Assumption~G2-a is satisfied with the LSD of $\Sigma_p$ degenerated at $1$.   Hence, arguing as in Corollary~\ref{cor: IID} and taking $a$ degenerated at $1$, the LSD of 
\(
\sqrt{n/p}\,(\tilde G^{\mathrm{IID}}-\bI_p)
\)
is the standard semicircle law. This coincides with the LSD of 
\(
(3/2)\sqrt{n/p}\,\mathbf{T}
\)
given in Theorem~2.5(1) of \cite{dornemann2025ties}.

\end{remark}


\paragraph{Verification of (c2) for $\bT_1$} \label{sec: pg2}
\noindent Let $\Pi = \tr\big(\sqrt{np^{-1}}(\bG-g_1\bI_p)\big)^R$. Then $\bE(\Pi - \bE(\Pi))^4  = A_1 - 4A_2A_4 + 6A_3A_4^2 - 3A_4^4$ where $A_i = \bE(\Pi^{5-i})$ for all $1 \leq i \leq 4$. Note that
\begin{eqnarray}
A_1 &=& \frac{n^{2R}}{p^{2R}}\sum_{\substack{k_{ur} \in [p],   \\ r \in [R],  \ u \in [4]}} \bE\Big(\prod_{\substack{r \in [R], \ u \in [4]}} \big(G_{k_{ur} k_{u\sigma_R(r)}}-g_1\delta_{k_{ur}k_{u\sigma_R}(r)}\big) \Big) \nonumber \\
&=& \dfrac{n^{2R}}{p^{2R}(n(n-1))^{4R}} \sum_{\substack{k_{ur} \in [p], \  i_{ur}, j_{ur} \in [n] \\ r \in [R], \  u \in [4]}} \bE\Big(\prod_{\substack{r\in [R], \  u \in [4]}}  \big(Y_{k_{ur},i_{ur},j_{ur}}Y_{k_{u\sigma_R(r)},i_{ur},j_{ur}}- \frac{(n(n-1))}{n^2}g_1\delta_{k_{ur}k_{u\sigma}(r)}\big)\Big).\nonumber
\end{eqnarray}
Let $K_{u2} = \{k_{ur}:\ r \in [R]\}$, $I_{u5}=\{i_{ur}:\  r \in [R]\}$ and $E_{u1} =
\{(k_{u(r+\epsilon_r)},i_{ur}):\ r \in [R],\ \epsilon_r = 0,1\}$.
Consider any connected bipartite graph between the distinct $k$-indices  $\cup_{u \in [4]} K_{u2}$  and $i$-indices $\cup_{u \in [4]} I_{u5}$ with edges $\cup_{u \in [4]} E_{u1}$. Let $\tilde{E}$ be the subset of $\cup_{u \in [4]} E_{u1}$ where each edge appears at least twice. For any subset $E$ of $\cup_{u \in [4]} E_{u1}$, define
\begin{eqnarray}
    A_{1E} &=& (n(n-1))^{-4R} \sum_{\substack{k_{ur} \in [p], \ i_{ur}, \ j_{ur} \in [n] \\ r \in [R], \  u \in [4]}} \bE\Big(\prod_{\substack{r\in [R], \ u \in [4]}}  Y_{k_{ur},i_{ur},j_{ur}}Y_{k_{u\sigma_R(r)},i_{ur},j_{ur}}\Big) \bI({E}). \nonumber 
\end{eqnarray}
Clearly, $A_{1(\cup_{u \in [4]} E_{u1} \setminus\tilde{E})}=0$ and hence $A_1 = A_{1\tilde{E}}+A_{1(\cup_{u \in [4]} E_{u1} \setminus\tilde{E})}=A_{1\tilde{E}}$. Now consider the following subsets of $\tilde{E}$:
\begin{eqnarray}
\tilde{E}_{1} &=& \text{set of all edges in $\tilde{E}$ where ${E}_{u1}$'s are disconnected for all $u \in [4]$}, \nonumber \\
\tilde{E}_{2,((s_1,s_2),s_3,s_4)} &=& \text{set of all edges in $\tilde{E}\setminus \tilde{E}_{1}$ where $E_{s_11}\cup E_{s_21}$, $E_{s_31}$ and $E_{s_41}$} \nonumber \\
&& \text{are disconnected for all $s_1 \neq s_2 \neq s_3 \neq s_4$}, \nonumber\\
\tilde{E}_{2} &=& \tilde{E}_{2,((1,2),3,4)} \cup \tilde{E}_{2,((1,3),2,4)}\cup \tilde{E}_{2,((1,4),2,3)} \cup  \tilde{E}_{2,((2,3),1,4)} \cup \tilde{E}_{2,((2,4),1,3)} \cup \tilde{E}_{2,((3,4),1,2)} \nonumber 
\end{eqnarray}
\begin{eqnarray}
\tilde{E}_{13,((s_1,s_2,s_3),s_4)} &=& \text{set of all edges in $\tilde{E}\setminus (\tilde{E}_{1}\cup \tilde{E}_{2})$ where $E_{s_11}\cup E_{s_21}\cup E_{s_31}$ and $E_{s_41}$} \nonumber \\
&& \text{are disconnected for all $s_1 \neq s_2 \neq s_3 \neq s_4$}, \nonumber \\
\tilde{E}_{3} &=& \tilde{E}_{3,((1,2,3),4)} \cup \tilde{E}_{3,((1,2,4),3)} \cup \tilde{E}_{3,((1,3,4),2)} \cup \tilde{E}_{3,((2,3,4),1)}, \nonumber\\
\tilde{E}_{3,((s_1,s_2),(s_3,s_4)} &=& \text{set of all edges in $\tilde{E}\setminus (\tilde{E}_{1}\cup \tilde{E}_{2}\cup \tilde{E}_{3})$ where $E_{s_11}\cup E_{s_21}$ and $ E_{s_31} \cup E_{s_41}$} \nonumber \\
&& \text{are disconnected for all $s_1 \neq s_2 \neq s_3 \neq s_4$}, \nonumber \\
\tilde{E}_{4} &=& \tilde{E}_{4,((1,2),(3,4))} \cup \tilde{E}_{4,((1,3),(2,4))} \cup \tilde{E}_{4,((1,4),(2,3))}, \nonumber\\
\tilde{E}_{5} &=& \text{set of all edges in $\tilde{E}$ where $E_{u1}$'s are all connected for $u \in [4]$.} \nonumber 
\end{eqnarray}
Observe that $\tilde{E} = \cup_{s=1}^5 \tilde{E}_{s}\ \ \ \text{and}\ \ \ A_1 = A_{1\tilde{E}} = \sum_{s=1}^{5} A_{1\tilde{E}_s}$, 
    \( A_{1\tilde{E}_{2,((1,2),3,4)}} = A_{1\tilde{E}_{2,((1,3),2,4)}}  = A_{1\tilde{E}_{2,((1,4),2,3)}} = A_{1\tilde{E}_{2,((2,3),1,4)}} =  A_{1\tilde{E}_{2,((2,4),1,3)}} = A_{1\tilde{E}_{2,((3,4),1,2)}},\)  
    \( A_{1\tilde{E}_{3,((1,2,3),4)}} = A_{\tilde{E}_{3,((1,2,4),3)}} = A_{\tilde{E}_{3,((1,3,4),2)}} = A_{\tilde{E}_{3,((2,3,4),1)}}, \ 
     \  A_{1\tilde{E}_{4,((1,2),(3,4))}} = A_{1\tilde{E}_{4,((1,3),(3,4))}} = A_{1\tilde{E}_{4,((1,4),(2,3))}}. \) 
Define $A_{12} = A_{1\tilde{E}_{2,((1,2),3,4)}}A_4^{-2}$  and $A_{13} = A_{1\tilde{E}_{3,((1,2,3),4)}}A_4^{-1}$. Clearly, we have $A_{1\tilde{E}_{1}} = A_4^4$ and $A_{1\tilde{E}_{4,((1,2),(3,4))}} = A_{12}^2$.  
Hence,
   \( A_{1} = A_4^4 + 6A_{12}A_4^2 + 4A_{13}A_4 + 3A_{12}^2 + A_{1\tilde{E}_5}. \) 
Similarly, one can easily show that $A_2 = A_4^3 + 3A_{12}A_4 + A_{13},\ \ A_3 = A_4^2 + A_{12}$.
Therefore, 
\begin{eqnarray}
    \bE(\Pi - \bE(\Pi))^4  &=& A_1 - 4A_2A_4 + 6A_3A_4^2 - 3A_4^4 = A_4^4 + 6A_{12}A_4^2 + 4A_{13}A_4 + 3A_{12}^2 + A_{1\tilde{E}_5} \nonumber \\
    && - 4(A_4^3 + 3A_{12}A_4+A_{13})A_4 + 6(A_4^2+A_{12})A_4^2 - 3A_4^4 = 3A_{12}^2 + A_{1\tilde{E}_5}. \label{eqn: g2} 
\end{eqnarray}
Let $\tilde{\tilde{E}}$ be the subset of $\cup_{u \in [4]} E_{u1}$ where each edge appears exactly twice. By the same argument as in the proof of (c1), we have
 \(   A_{12}^2 = A_{1(\tilde{E}_{4,((1,2),(3,4))} \cap \tilde{\tilde{E}})} + O(p^2)\ \  \text{and}\ \ A_{1\tilde{E}_5} = A_{1(\tilde{E}_5 \cap \tilde{\tilde{E}})} + O(p). \) 
Define
\begin{eqnarray}
\tilde{A}_{1,(r_1,r_2,r_3r_4)} &=&  \dfrac{n^{2R}}{p^{2R}(n(n-1))^{4R}} \sum_{\substack{k_{ur} \in [p], \ i_{ur}, \ j_{ur} \in [n] \\ (r \in [R], \ u \in [4])}} \bE\Big(\prod_{\substack{r\in [R], \ u \in [4]}}  \big(Y_{k_{ur},i_{ur},j_{ur}}Y_{k_{u\sigma_R(r)},i_{ur},j_{ur}}-\frac{n(n-1)}{n^2}g_1\delta_{k_{ur}k_{u{\sigma_R}(r)}}\big)\Big) \nonumber \\
&& \hspace{2.5 cm}\bI((k_{1r_1}, i_{1r_1}) = (k_{2r_2},i_{2r_2}),(k_{3r_3}, i_{3r_3}) = (k_{4r_4},i_{4r_4})), \nonumber \\
\tilde{A}_{2,(r_1,r_2,r_3r_4)} &=&  \dfrac{n^{2R}}{p^{2R}(n(n-1))^{4R}} \sum_{\substack{k_{ur} \in [p], \ i_{ur}, \ j_{ur} \in [n] \\ (r \in [R], \ u \in [4])}} \bE\Big(\prod_{\substack{r\in [R], \ u \in [4]}}  \big(Y_{k_{ur},i_{ur},j_{ur}}Y_{k_{u\sigma_R(r)},i_{ur},j_{ur}}-\frac{n(n-1)}{n^2}g_1\delta_{k_{ur}k_{u{\sigma_R}(r)}}\big)\Big) \nonumber \\
&& \hspace{2 cm}\bI((k_{1\sigma_{R}(r_1)}, i_{1r_1}) = (k_{2r_2},i_{2r_2}),(k_{3r_3}, i_{3r_3}) = (k_{4r_4},i_{4r_4})), \nonumber \\
\tilde{A}_{3,(r_1,r_2,r_3r_4)} &=&  \dfrac{n^{2R}}{p^{2R}(n(n-1))^{4R}} \sum_{\substack{k_{ur} \in [p], \ i_{ur}, \ j_{ur} \in [n] \\ (r \in [R], \ u \in [4])}} \bE\Big(\prod_{\substack{r\in [R], \ u \in [4]}}  \big(Y_{k_{ur},i_{ur},j_{ur}}Y_{k_{u\sigma_R(r)},i_{ur},j_{ur}}-\frac{n(n-1)}{n^2}g_1\delta_{k_{ur}k_{u{\sigma_R}(r)}}\big)\Big) \nonumber \\
&& \hspace{1.5 cm}\bI((k_{1\sigma_{R}(r_1)}, i_{1r_1}) = (k_{2r_2},i_{2r_2}),(k_{3\sigma_{R}(r_3)}, i_{3r_3}) = (k_{4r_4},i_{4r_4})). \nonumber
\end{eqnarray}
Clearly,
\(    A_{1(\tilde{E}_{4,((1,2),(3,4))} \cap \tilde{\tilde{E}})} = \sum_{\substack{r_u \in [R], (u \in [4]) } } (\tilde{A}_{1,(r_1,r_2,r_3r_4)} + 2\tilde{A}_{2,(r_1,r_2,r_3r_4)}+\tilde{A}_{3,(r_1,r_2,r_3r_4)}). \) 
Denote by $${\mathbb{\tilde{G}}}(k_r,k_{{\sigma_R}(r)})= \sqrt{np^{-1}}(\bG(k_r,k_{\sigma_R(r)})-g_1\delta_{k_r,k_{{\sigma_R}(r)}}).$$  
Note that
\begin{eqnarray}
    \tilde{A}_{1,(r_1,r_2,r_3r_4)} &=& \frac{n^2}{p^2(n(n-1))^{4}}\sum_{\substack{k_{ur_u},k_{u(r_u+1)}, \\ k_{(u+1)(r_u+1)} \in [p]\\ u=1,3} } \sum_{\substack{i_{ur_u},j_{ur_u},\\ j_{(u+1)r_u} \in [n]\\ u=1,3}}  \prod_{u=1,3}\bE(\mathbb{\tilde{G}}^{R-1}(k_{ur_u},k_{u(r_u+1)}) \nonumber \\
     && \hspace{1 cm}Y_{k_{u(r_u+1)},i_{ur_u},j_{ur_u}}Y_{k_{(u+1)(r_u+1)},i_{ur_u},j_{(u+1)r_{u+1}}}\mathbb{\tilde{G}}^{R-1}(k_{(u+1)(r_u+1)},k_{ur_u})) \bE(Y_{k_{ur_u},i_{ur_u},j_{ur_u}}Y_{k_{ur_u},i_{ur_u},j_{(u+1)r_{u+1}}}) \nonumber \\
        &=& \frac{n^2}{p^2(n(n-1))^{4}}\sum_{\substack{k_{ur_u},k_{u(r_u+1)}, \\ k_{(u+1)(r_u+1)} \in [p]\\ u=1,3}}  \sum_{\substack{i_{ur_u},j_{ur_u} \in [n] \\ u=1,3}} \prod_{u=1,3}\bE(\mathbb{\tilde{G}}^{R-1}(k_{ur_u},k_{u(r_u+1)}) \nonumber \\
          && \hspace{1 cm}Y_{k_{u(r_u+1)},i_{ur_u},j_{ur_u}}Y_{k_{(u+1)(r_u+1)},i_{ur_u},j_{ur_u}}\mathbb{\tilde{G}}^{R-1}(k_{(u+1)(r_u+1)},k_{ur_u}))  \bE(Y_{k_{ur_u},i_{ur_u},j_{ur_u}}^2) \nonumber 
         \end{eqnarray}
\begin{eqnarray}
    & \lesssim & \frac{n^2}{p^2(n(n-1))^{4}}\sum_{\substack{k_{ur_u},k_{u(r_u+1)}, \\ k_{(u+1)(r_u+1)} \in [p]\\ u=1,3}}  \sum_{\substack{i_{ur_u},j_{ur_u} \in [n]\\ u=1,3}}\prod_{u=1,3}\bE(\mathbb{\tilde{G}}^{R-1}(k_{ur_u},k_{u(r_u+1)}) \nonumber \\
    && (Y_{k_{u(r_u+1)},i_{ur_u},j_{ur_u}}Y_{k_{(u+1)(r_u+1)},i_{ur_u},j_{ur_u}}-g_1\delta_{k_{u(r_u+1)}k_{(u+1)(r_u+1)}}+g_1\delta_{k_{u(r_u+1)}k_{(u+1))(r_u+1)}})\tilde{\mathbb{G}}^{R-1}(k_{(u+1)(r_u+1)},k_{ur_u})) \nonumber \\
    & = & \frac{{n}}{{p}(n(n-1))^{2}}\sum_{\substack{k_{ur_u},k_{u(r_u+1)}, \\ k_{(u+1)(r_u+1)} \in [p]\\ u=1,3}} \prod_{u=1,3} \big[\bE(\tilde{\mathbb{G}}^{R-1}(k_{ur_u},k_{u(r_u+1)})\mathbb{\tilde{G}}(k_{u(r_u+1)}, k_{(u+1)(r_u+1)})\tilde{\mathbb{G}}^{R-1}(k_{(u+1)(r_u+1)},k_{ur_u}))\big]\nonumber\\
    &+& \frac{{n^2}}{{p}^2(n(n-1))^{2}}\sum_{\substack{k_{ur_u},k_{u(r_u+1)}, \\ k_{(u+1)(r_u+1)} \in [p]\\ u=1,3}} \prod_{u=1,3}\big[\bE(\mathbb{\tilde{G}}^{R-1}(k_{ur_u},k_{u(r_u+1)})g_1\delta_{k_{u(r_u+1)}k_{(u+1))(r_u+1)}}\mathbb{\tilde{G}}^{R-1}(k_{(u+1)(r_u+1)},k_{ur_u}))\big]\nonumber \\
    &\lesssim& \Bigg(\frac{\sqrt{n}}{\sqrt{p}(n(n-1))} \bE {\rm{\tr}}(\mathbb{\tilde{G}}^{2R-1})\Bigg)^2+\Bigg(\frac{n}{p(n(n-1))}\bE {\rm{\tr}}(\mathbb{\tilde{G}}^{2R-2}))\Bigg)^2 = O(1). \nonumber 
\end{eqnarray}
Similarly, one can show that $\tilde{A}_{u,(r_1,r_2,r_3r_4)} = O(1)$ for $u=2,3$. 
Similar arguments prove that  $A_{1(\tilde{E}_5 \cap \tilde{\tilde{E}})} = O(p^2)$. 
Hence we have $\mathbb{E}(\Pi-\bE(\Pi))^4 = O(p^2)$. This establishes (c2) and completes the proof Theorem \ref{thm: 1}.

\vskip 5pt
\noindent \textbf{Acknowledgment}. The authors sincerely thank the Editor, the Associate Editor, and the Referees for their careful reading of the manuscript and for their constructive comments and suggestions. Their insightful feedback has significantly improved the clarity, presentation, and overall quality of the paper.

  \bibliographystyle{elsarticle-num-names} 
  \bibliography{kendallref}

\end{document}